\documentclass[12pt]{article}

\usepackage[T1]{fontenc}
\usepackage[T1]{fontenc}
\usepackage{amssymb,amsmath,amsthm,mathrsfs}
\usepackage{tikz}
\usepackage{pgf}
\usepackage{graphics}
\usepackage[active]{srcltx}

\newtheorem{teo}{Theorem}
\newtheorem{cor}{Corollary}
\newtheorem{prop}{Proposition}
\newtheorem{dfn}{Definition}
\newtheorem{lm}{Lemma}

\newtheorem*{teoa}{Theorem A}
\newtheorem*{teoa'}{Theorem A'}
\newtheorem*{teob}{Theorem B}
\newtheorem*{teob'}{Theorem B'}
\newtheorem*{teoc}{Theorem C}
\newtheorem*{teoc'}{Theorem C'}
\newtheorem*{teocs}{Theorem B' (symplectic version of theorem B)}
\newtheorem*{teod}{Theorem D}
\newtheorem*{teod'}{Theorem D'}

\title{Dynamical Properties of Gaussian Thermostats}
\author{Ivana  Latosinski\,\,\,\,\,\,\,\,\,Enrique Pujals }

\begin{document}
\maketitle
\abstract
In this work we show that the set of \emph{Kupka-Smale}  Gaussian thermostats  on a compact
manifold is generic. A Gaussian thermostat is  \emph{Kupka-Smale} if the closed orbits
are hyperbolic and the heteroclinic intersection are transversal.

We also show a dichotomy between robust transitivity and existence 
of arbitrary number of attractors or repellers orbits. The main tools are the concept of transitions
adapted to the conformally symplectic context and a perturbative theorem which is a version 
of the Franks lemma for Gaussian thermostats.

Finally we provide some conditions in terms of geometrical invariants for an invariant set of a Gaussian thermostat to have dominated splitting. From that we conclude some dynamical properties for the surface case.

\tableofcontents

\section{Introduction}

Gaussian thermostats or isokinetic dynamics were introduced by Hoover \cite{h} who considered a class of mechanical dynamical systems with forcing and thermostatting term based on the Gauss least Constraint Principle for nonholonomic constraints.

Let $(M, g)$ be a Riemannian manifold and let $E:M\longrightarrow TM$ be a vector field in $M$.
We call Gaussian thermostat the flow $\phi: TM \longrightarrow TM$
in which an orbit $\eta(t)=(x(t),v(t))$ satisfies the equation 

$$\left\{\begin{array}{rcl}
\dot{x}&=&v\\
\\
\nabla_{v}v &=& E - \frac{ g(E, v )}{g( v,v )}v
\end{array}\right.$$
where $\nabla$ is the Riemannian connection of  $(M, g)$.

The thermostat is the parcel $-\frac{ g(E, v )}{g( v,v )}v$ of the second equation.
It subtracts the component of $E$ which is not parallel to the orbit tangent vector $v$,
changing only the direction of $v$ and keeping its modulus constant.

Gaussian thermostats have been  proposed as  models for systems out of equilibrium  in statistical mechanics as it is discussed in the papers by Gallavotti and Ruelle \cite{g}, \cite{r}, \cite{gr}. Gauss least Constraint Principle is particularly useful in describing the motion of constrained systems. No fundamental rules or variational principles are avaliable for such constraints and it is not true that the work performed by the constraints should be a minimum.  Such flows require special methods to compensate for the natural dissipation of work into heat. For simplicity, it is convenient to remove the heat in such a way that the nonequilibrium state is a steady one. By steady we mean some state variables are held constant and for the Gaussian thermostat we maintain the kinetic energy constant. 

As an example of Gaussian thermostat provided in \cite{w1}, let $N$ be  a Riemannian manifold and consider $M=\mathbb{S}^{1} \times N$ equipped with the product metric. 
Let $E$ a vector field tangent to $\mathbb{S}^{1}$ with constant modulus $| E |$. For $v \in SM$, write $v=v_{0}+v_{1}$ the decomposition of $v$
in a component parallel to $\mathbb{S}^{1}$ and $N$, respectively. The equation of the Gaussian thermostat on the component $v_{0}$ is $v_{0}=|E|(1-v_{0}^{2})$
and the orbits of the Gaussian thermostat are parallel lines tangent to $\mathbb{S}^{1}$ or have this component equal as $v_{0}(t)=tanh(|E|t+\lambda_{0})$.
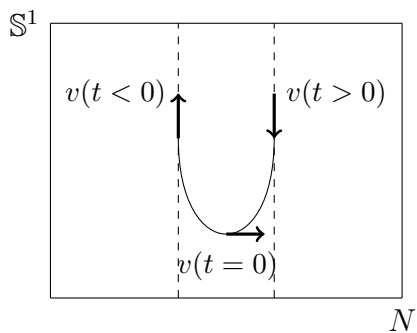
\begin{figure}[h]
\begin{center}
\begin{tikzpicture}[scale=0.85]
\draw (-1,-0.5) -- (4.5,-0.5) node[anchor=north]{$N$};
\draw (-1,-0.5) -- (-1,3.8) node[anchor=east]{$\mathbb{S}^{1}$};
\draw (-1,3.8) -- (4.5,3.8);
\draw (4.5,-0.5) -- (4.5,3.8);
\draw [dashed](1,-0.5) -- (1,3.8);
\draw [dashed](2.5,-0.5) -- (2.5,3.8);
\draw (1,2) .. controls (1,0) and (2.5,0) .. (2.5,2);
\draw [->,very thick](1,2) -- (1,2.7) node[anchor=east]{\small{$v(t<0)$}};
\draw [->,very thick](1.75,0.5) -- (2.35,0.5); 
\node at (2,0) {\small{$v(t=0)\quad$}};
\draw [<-,very thick](2.5,2) -- (2.5,2.7) node[anchor=west]{\small{$v(t>0)$}};
\end{tikzpicture}
\end{center}
\caption{Typical orbit  over $\mathbb{S}^{1}\times N$.}
\end{figure}
This Gaussian thermostat has a global attractor $A$ and a global repeller $R$, both normally hyperbolic. 
They are given by
$A = \{\xi=\xi_{0}+\xi_{1}  \in SM | \xi_{0} = \frac{E}{| E |} \}$ and 
$R = \{\xi=\xi_{0}+\xi_{1}  \in SM | \xi_{0} = -\frac{E}{| E |} \}$ .

If now we consider $N=\mathbb{S}^{1}$ and change the vector field $E$ to make an irrational angle with $N$ then we still have two invariant sets, both normally hyperbolic and the dynamics on each one is conjugated to a linear irrational flow on $\mathbb{S}^{1} \times \mathbb{S}^{1}$.
We call each torus with these properties an irrational torus. 

The Hamiltonian formulation of the Gaussian thermostat was obtained by Dettmann and Morris \cite{dm}.
This also connects us with the conformally symplectic dynamics \cite{w2}. 
A Gaussian Thermostat can be an example of a conformally Hamiltonian system.
These systems are more general than Hamiltonian and symplectic systems but are closely related to them.
They are determined by a non-degenerate $2$-form $\Omega$ on the phase space and a function H, called again a Hamiltonian.  The form $\Omega$ is not assumed to be closed but $d\Omega = \gamma \land \Omega$ for some closed $1$-form $\gamma$. This condition guarantees, at least locally, the form $\Omega$  can be multiplied by a nonzero function to give a symplectic structure. The skew-ortogonality of tangent vectors is preserved under multiplication
of the form by any nonzero function, hence the name conformally symplectic structure. 
% 
% A Gaussian thermostat exhibits rigidity on the Lyapunov spectrum  \cite{w2} and is  non conservative. 
% In this setting we can construct examples of  Axiom A flows and with dominated splitting but not partially hyperbolic flows.

A geometrical perspective for Gaussian thermostats was introduced  by Wojtkowski.
He shows that the Gaussian thermostat coincide with  W-flow, a modification of the geodesic flow on a Weyl manifold and, with this approach, 
he estabilish some connections between negative curvature of the Weyl structure and the hyperbolicity of W-flows \cite{w1}. Latter, in \cite{w3} he proposed that  Gaussian thermostats are geodesic flows of a special metric connections which is non-symmetric but has isometric parallel transport.

In the present work we approach Gaussian thermostats from a dynamical system point of view and we discuss some of their generic properties as Kupka-Smale; after that, we propose dynamical dichotomies adapting to the Gaussian thermostats context, the results in \cite{bdp}. Moreover, extending results in \cite{w1} we provides sufficient conditions for weak form of hyperbolicity based on the curvature.

%In this work we prove 
For the first part of this work we restrict ourselves to a subset of the external forces $E$. More precisely, we fixed a Riemmanian metric $g$ and we consider perturbations only by the action of an external force. Following that, we define $\mathscr{X}_{g}(M) \subset \mathscr{X}^{\infty}(M)$
as the set of vector fields $E$ such that the $1$-form $\gamma(.)= g(E,.)$ is closed.

We say that a Gaussian thermostat $(M,g,E)$ on a compact manifold $M$ is \emph{Kupka-Smale}
if it satisfies:
\begin{itemize}
\item[(i)] all closed orbits are hyperbolic,
\item[(ii)] all heteroclinic intersections are transversal.
\end{itemize}
And we also say a property $P$ is generic in $\mathscr{X}_{g}(M)$ if there exists a residual subsets in
$\mathscr{X}_{g}(M)$ which satisfies $P$.

The first result in this work is 

\begin{teoa}
The set of vector fields in $\mathscr{X}_{g}(M)$ 
that induce Kupka-Smale  Gaussian thermostats 
is generic in $\mathscr{X}_{g}(M)$.
\end{teoa}

We also prove a dichotomy between robust transitivity and 
the existence of an arbitrary number of attractor or repeller orbits
for the Gaussian thermostat.

Given a flow $\phi$ and a hyperbolic periodic saddle orbit $\eta$, 
we define the \emph{homoclinic class} $H(\eta,\phi)$
as the closure of the set of transversal intersection of the unstable and stable manifolds of $\eta$.
$$
H(\eta, \phi) = \overline{W^{u}_{\phi}(\eta)\pitchfork W^{s}_{\phi}(\eta)}.
$$

Let $M$ be a differential manifold with  $dim (M)=2n+1$, let $F: M \to TM$ be a vector field such that 
$F(x) \neq 0$ for all $x \in M$
and $\phi$ the associated flow. 
We denote $\hat{T}M$ the quotient of the tangent bundle $TM$ by the vector field $F$.
Let $\pi : TM \to \hat{T}M$  be the canonical projection of $v \in
T_{x}M$ in $\hat{T}_{x}M$. The composition $
A^{t}_{x} = \pi \circ d\phi^{t}_{x} : \hat{T}_{x}M \to
\hat{T}_{\phi^{t}(x)}M$ defines the \emph{transversal derivative cocycle associated to $\phi$}
or the \emph{linear Poincar\'e application}.
We say that an invariant set $\Lambda$ has \emph{dominated splitting (on the linear Poincar\'e application)}
if $\hat{T}_{\theta}M = N^{cs}_{\theta}\oplus N^{cu}_{\theta}$
is a invariant decomposition of the transversal derivative cocycle $T^{t}$ 
and there exists constants $\lambda, C >0$ such that for $t>0$ and  for all $\theta \in \Lambda$
$$
\| T^{t}|_{N^{cs}_{\theta}} \| \leq C e^{-\lambda t} m (T^{t}|_{ N^{cu}_{\theta}})
$$
where $m ( L )=min \{ \| Lv \| : \| v \| = 1 \}$.

We say that an invariant set $\Lambda$ is hyperbolic, if the linear Poincar\'e flow has a decomposition $\hat{T}_{\theta}M = N^{s}_{\theta}\oplus N^{u}_{\theta}$ and any vector in $N^{s}$ is uniformly contracted by forward iterates and any vector in $N^u$ is uniformly contracted for backward iterates.

In both examples described before, the attracting and repelling normally hyperbolic torus exhibits a dominated splitting and they are not hyperbolic.
%\end{dfn}

In the next theorem we proved a dichotomy in the $C^1-$cathegory for homoclinic classes:

\begin{teob}\label{appfrankstg}
Let $E \in \mathscr{X}_{g}(M)$ and  
$\eta$ be a hyperbolic saddle periodic orbit 
for the Gaussian thermostat $(M,g,E)$ 
which flow is $\phi$.

One of the alternatives is true:
\begin{itemize}
\item[(i)] the homoclinic class $H(\eta,\phi)$ has a dominated splitting,

\item[(ii)] given a small neighborhood $V$ of  $  H(\eta, \phi) \subset V$, $k \in \mathbb{N}$, and $\varepsilon > 0$, exists
$\tilde{E} \in \mathscr{X}_{g}(M)$ with $d( \tilde{E},  E)_{C^{1}} < \varepsilon$
such that the Gaussian thermostat $(M,g,\tilde{E})$
has $k$ attractor or repeller periodic orbits 
in $V$. 
\end{itemize}
\end{teob}

For the case of surfaces, it is possible to get the following theorem which is a local version of theorem 5.2 in \cite{w1}.

\begin{teob'}\label{appfrankstg'} Let  $M$ be a Riemannian surface.
Let $E \in \mathscr{X}_{g}(M)$ and  
$\eta$ be a hyperbolic saddle periodic orbit 
for the Gaussian thermostat $(M,g,E)$ 
which flow is $\phi$.

One of the alternatives is true:
\begin{itemize}
\item[(i)] generically, the homoclinic class $H(\eta,\phi)$ is hyperbolic,

\item[(ii)] given a small neighborhood $V$ of  $  H(\eta, \phi) \subset V$, $k \in \mathbb{N}$, and $\varepsilon > 0$, exists
$\tilde{E} \in \mathscr{X}_{g}(M)$ with $d( \tilde{E},  E)_{C^{1}} < \varepsilon$
such that the Gaussian thermostat $(M,g,\tilde{E})$
has $k$ attractor or repeller periodic orbits 
in $V$. 
\end{itemize}
\end{teob'}

The main tools for the proof of Theorem B are the concept of transitions 
adapted to the conformally symplectic context and the perturbative theorem, stated hereafter. 
\begin{teoc}
Given $E \in \mathscr{X}_{g}(M)\cap \mathscr{X}^{r}(M)$, $4 \leq r \leq \infty$, $\theta$ 
a point in a periodic orbit $\eta$ of a Gaussian thermostat $(M,g,E)$, and 
$T : \hat{T}_{\theta}SM \to \hat{T}_{\theta}SM$ the transversal derivative cocycle
along $\eta$.

Given $\varepsilon> 0$ there exists
$\delta>0$ such that for any  conformally symplectic cocycle  $L$ such that $\| L - T \| < \delta$ then 
there exists $\tilde{E} \in \mathscr{X}_{g}(M)$ with $d( \tilde{E},  E)_{C^{1}} < \varepsilon$
which defines a Gaussian thermostat $(M,g,\tilde{E})$
such that $\eta$ is an orbit of it and the transversal derivative cocycle
associated to $\theta$ is $L$.
\end{teoc}

In  \cite{franks}, Franks proved that given 
a $C^{1}$ diffeomorphism $f: M \to M$  
over a Riemannian manifold $(M, g)$ and
 $\varepsilon > 0$, if we take a periodic point $x \in M$,
we can perform a $C^{1}$ small perturbation $g$ of $f$ such that
$g^{n}(x)=f^{n}(x)$, $n \in \mathbb{Z}$, and $dg^{n}$ is any isomorphism 
$\varepsilon$-close of $df^{n}$, for $n \in \mathbb{Z}$.
This result is known as \emph{Franks lemma}. The key point in that lemma is since is only required that $g$ is $C^1-$close to $f$, the support of the perturbation can be done arbitrary small in such a way that the perturbation preserves the trajectory.

In our context, a perturbation on the Gaussian thermostat is a perturbation on  the vector field $E$. This means that perturbations are not local, in fact,
a perturbation on  the vector field  implies a
perturbation on a cylinder on the tangent space where the Gaussian thermostat flow is defined. Therefore, even the support of the perturbation of the vector fiel $E$  can be done arbitrary small, this produces a perturbation on the Gaussian thermostat that is not localized.

For geodesic flows, a similar problem appears, since
geodesic flow perturbations are metric perturbations and these perturbations  are not local as was described above. In fact, to perturb a metric on a neighborhood of a closed geodesic
means that it is performed a perturbation in a cylinder in the tangent space 
where the geodesic flow is defined.
However, and overcoming such difficulty, in \cite{contreras} Contreras proved a version of the Franks lemma
for the geodesic flow. A similar version for the particular case of Magnetic flows have been proved in \cite{miranda}.

As we said, for a Gaussian thermostat the situation is the same as the geodesic flow.
In this case, we keep the metric unperturbed and we perturb the  vector field $E$, 
and this 
implies a perturbation over a cylinder in the tangent space
where the flow is defined.

To perform Franks lemma for Gaussian thermostats we adapt the ideas of Contreras in \cite{contreras} keeping in mind that the metric is untouched and any perturbation is done on the external vector field.

We complete this work relating some geometrical properties with dynamical 
 ones for the Gaussian thermostat in the following theorems.
 %, considering  the 1-form $\gamma(\cdot) = \langle E, \cdot \rangle$.

\begin{teod}
Let $\Lambda \subset SM$ be an invariant subset of a $C^2$ Kupka-Smale Gaussian thermostat flow with non-positive  sectional curvature of the Weyl structure and  for any $\theta = (x,v) \in \Lambda$ we have $\gamma (v)>0$. Then $\Lambda$ has dominated splitting. 
\end{teod}

Using the results in \cite{aubin}  for non-singular three dimensional flows which are based on the theorems in \cite{pujalssambarino1}, we can clonclude the following.

\begin{teod'}
Let $M$ be a $2$-dimensional manifold and $\Lambda \subset SM$ be an invariant subset of a $C^2$ Kupka-Smale Gaussian thermostat flow with non-positive  curvature of the Weyl structure and for any $\theta = (x,v) \in \Lambda$ we have $\gamma (v)>0$. Then $\Lambda = \tilde{\Lambda} \cup \mathcal{T}$, where $\tilde{\Lambda}$ is hyperbolic and  $\mathcal{T}$ is a finite union of normally hyperbolic irrational tori. 
\end{teod'}

We want to point out that for Gaussian thermostat both type of sets listed in the conclusion of theorem D'  can  exist. In the beginning of the present section we show  that there are examples of normally hyperbolic invariant torus. Also, given a metric with negative curvature (and so the geordesic flow is Anosov) it is possible to chose an external field $E$ such that the  Gaussian thermostat is a Derived of Anosov and in particular has an attracting hyperbolic set.

We start our work in the next section in which we make an introduction to conformally symplectic dynamics and its properties.
Then, we show the conformally Hamiltonian approach to the  Gaussian thermostats. We also construct some perturbations used along the work and good coordinates to work with.
The remaining  is dedicated to the proofs of the theorems.  The first one, which says the set of vector fields in $\mathscr{X}_{g}(M)$ 
which defines Kupka-Smale  Gaussian thermostats  is generic in $\mathscr{X}_{g}(M)$.
The second corresponds to the dichotomy between dominated splitting and the existence of attractors/repellers on the neighborhood of homoclinical classes for  Gaussian thermostats.
 Then we prove the third theorem, the perturbative theorem, which was used as a tool for the proof of the previous one.  
 Finally, in the last section we prove the relantionship between geometrical
 and dynamical  properties.

\section{Conformally symplectic dynamics}

In this section, we introduce the conformally symplectic dynamics.
This systems are more general than the symplectic ones and deeply related to these.
A conformally symplectic system is determined by a a manifold $M$
and a $2$-form $\Omega$ over $M$,
non degenerated and which satisfies $d\Omega = \gamma
\land \Omega$ with $\gamma$ a closed $1$-form. 
The closeness of $\gamma$ guarantees that, locally, $\Omega$
can be multiplied by a non null function to obtain a symplectic structure.

Let $\Omega_{0} = \sum_{i=1}^{n} dp_{i} \land dq_{i}$ the canonical symplectic  $2$-form in
$\mathbb{R}^{2n}$.

\begin{prop}[\cite{w2}]\label{grupocs}
Given a invertible linear application $S: \mathbb{R}^{2n} \to \mathbb{R}^{2n}$,
the following are equivalent:
\begin{enumerate}
\item $\Omega_{0}(Su,Sv) = \mu \Omega_{0}(u,v)$ for a scalar $\mu > 0$ and 
$u,v \in \mathbb{R}^{2n},$
\item $\Omega_{0}(Su,Sv)
= 0$ if and only if $\Omega_{0}(u,v) = 0,$
\item Lagrangian subspaces are invariant by $S.$
\end{enumerate}
\end{prop}

\begin{dfn} The set of matrices that satisfy the properties of the proposition \ref{grupocs} form a subgroup of $\text{GL}(\mathbb{R}^{2n})$, it is called a  \emph{conformally symplectic group} and is denoted by $\text{CS}(\mathbb{R}^{2n})$. 
\end{dfn}

\begin{dfn}
Given $\mathscr{B} = \{ e_{1}, \dots, e_{n}, e_{n+1}, \dots, e_{2n} \}$ a basis for 
$\mathbb{R}^{2n}$. 
We say that  $\mathscr{B}$ a \emph{canonical conformally symplectic basis} if
$\Omega(e_{i}, e_{j})=\delta_{i, j \text{ mod }(n)}$.
\end{dfn}

This basis coincides with the symplectic basis for $\mathbb{R}^{2n}$ and,
given a conformally symplectic application, is always possible to find
a conformally symplectic basis associated to this application.

\begin{prop}\label{conju}
If $\lambda$ is an eigenvalue of $S \in \text{CS}(\mathbb{R}^{2n})$ then
there exists $\mu \in \mathbb{R}$ such that $\frac{\mu}{\lambda}$ also is a eigenvalue of $S$.
\end{prop}

Let $X$ be  a metric space  and
$f: X \longrightarrow X$ be a continuous application. A continuous application 
$A: X \longrightarrow CS(\mathbb{R}^{2n})$ is a {\em a  conformally symplectic cocycle}
if $x$ is a periodic trajectory of $f$ (i.e. $f^k(x)=x$) then the matrix
$A_k=\Pi_{i=0}^{k-1} A_{f^i(x)}$ verifies proposition \ref{conju}. Furthermore, if all 
the eigenvalues has modulus different than one we say that $A_k$ is hyperbolic and the
periodic trajectory is hyperbolic. 
The next results extends proposition \ref{conju} to the Lyapunov exponent for conformally
symplectic cocycle.
Let $X$ be  a  measurable space with probability measure $m$ and
$f: X \longrightarrow X$ be an ergodic application. Let $A: X \longrightarrow
CS(\mathbb{R}^{2n})$ be a measurable application and consider the associated cocycle which is called a measurable conformally symplectic cocycle.
From Oseledets theorem, 

\begin{teo}[\cite{w2}]\label{oseletes}
If a measurable conformally symplectic cocycle
$A(x)$, $x \in X$, satisfies the integrability condition, i.e.,
$\int_{X} \log_{+} \| A(x) \| dm (x) < + \infty$,  
then the 
Lyapunov exponents 
$\lambda_{1}< ... <\lambda_{s}$ and the flag
$$
\{ 0 \} = \mathbb{V}_{0} \subset ...\subset \mathbb{V}_{s-1}
\subset \mathbb{V}_{s} = \mathbb{R}^{2n}
$$
flag are well defined. Moreover:

\begin{enumerate}
\item $\lambda_{k} + \lambda_{s-k+1} = b$ with $b = \int_{X} \log
| det A (x)| dm (x);$

\item the multiplicity of $\lambda_{k}$ and $\lambda_{s-k+1}$ 
are equal for $k=1,2,...,s;$

\item $\mathbb{V}_{s-k}$ is the orthogonal complement of 
$\mathbb{V}_{s}$ with respect to $\Omega_{0}$.
\end{enumerate}
\end{teo}
Conformally symplectic cocycles are going to be redefined and revisited in section
\ref{transl} in a more general frameworks. 
A particular case of a conformally derivative cocycle is the transversal derivative cocycle associated to the flow of Gaussian thermostat. 

Let $\Omega$ a $2$-form defined on the quotient
$\hat{T}M$, $\Omega_{x} : \hat{T}M \times \hat{T}M \to \mathbb{R}$, with $x \in M$  such that 
\begin{itemize}
\item[(i)] $\Omega$ is non degenerated,
\item[(ii)] there exists a closed $1$-form $\gamma$ defined on the quotient $\hat{T}M$, $\gamma_{x} : \hat{T}M \to \mathbb{R}$, such that $d\Omega = \gamma \land \Omega$,
\item[(iii)] there exists $\beta : M \longrightarrow \mathbb{R}$ such that
$$(\pi \circ \phi_{x}^{t})^{*}\Omega=e^{\int_{0}^{t}\beta(\phi^{s}_{x})ds}\Omega,\quad
\forall x \in M \text{ e } t \in \mathbb{R}.
$$
\end{itemize}

The pair $(M, \Omega)$ is called a conformally symplectic manifold.

\begin{dfn}
Let $(M, \Omega)$ be a conformally symplectic manifold. A $C^{\infty}$ function $H : M \to \mathbb{R}$ is called a Hamiltonian and 
the conformally Hamiltonian vector field is the uniquely defined vectors field $F$ provided by the relation $\Omega (., F)=dH$.
\end{dfn}

Restricted to $\hat{T}M$ , the $2$-form  $\Omega$ is conformally symplectic with 
$\beta(t) = \gamma(F(\phi^{t}))$
and the Hamiltonian $H$ is a first integral for the flow. 

Given an orbit segment $\eta$ of the conformally symplectic flow $\phi$ then 
there exist local coordinates such that the linear cocycle satisfies 
$$
(A^{t})^{*}JA^{t} = e^{\int^{t}_{0}\beta(\phi^{s}_{x}(s))ds}J,
$$
where {\small$J=\left( \begin{array}{cc}
0 & -I\\
I & 0
\end{array} \right)$} is an $2n \times 2n$ matrix and $I$ the $n \times n$ identity.
This cocycle has symmetry of the Lyapunov spectrum.

A matrix $Y$ is called \emph{infinitesimally conformally symplectic} if
$$
Y^{*}J + JY = vJ.
$$
This matrices can be written as 
\begin{displaymath}
Y = \left(\begin{array}{cc}
 \beta & \gamma \\
 \alpha & vI-\beta^{*}
\end{array}\right)
\end{displaymath}
with $\alpha$ and $\gamma$ symmetric.
The tangent space of a matrix
$X  \in CS(\mathbb{R}^{2n})$ at the energy level $v$
satisfies
$$T_{X}^{v^{2}}\text{CS}(\mathbb{R}^{n}) = X T_{I}\text{CS}(\mathbb{R}^{n})$$
with  $T_{I}\text{CS}(\mathbb{R}^{n})$ the tangent space of the identity.

Given a Gaussian thermostat $(M,g,E)$, a $\varepsilon$-perturbation of $(M,g,E)$
is a $\varepsilon$-perturbation of $E$, i.e., it is the Gaussian thermostat defined by 
$(M,g,\tilde{E})$ with $d(\tilde{E},E)< \varepsilon$
in the adequate topology.

\section{The Gaussian thermostat as a conformally symplectic Hamiltonian flow}
Given $(M,g,E)$ a Gaussian thermostat restricted to the energy level $c$.
The conformally Hamiltonian structure of this Gaussian thermostat is defined as follows.

Let
\begin{itemize}
\item[(i)] The $C^{\infty}$ function $H : TM \to \mathbb{R}$, the Hamiltonian, and defined by
$H(x,v) = \frac{1}{2} g(v,v).$
\item[(ii)] The tautological $1$-form in $TM$ induced by $g$: 
$\kappa=h^{-1}(\kappa^{*}).$
\item[(iii)] The canonical $2$-form in $TM$ induced by $g$: 
$\omega=-d\kappa.$
\item[(iv)] The $1$-form $\gamma$ defined in $TM$ by
$\gamma = g(E,.).$
\item[(v)] The non-degenerated $2$-form 
$\Omega = \omega -  \frac{1}{c}\gamma \land \kappa$. 
\end{itemize}

Then, restricted to a energy level $c$, the conformally Hamiltonian flow $(M,H,\Omega)$
associated to the vector field $E \in \mathscr{X}_{g}(M)$ coincides with the Gaussian thermostat $(M,g,E)$.

If $E \in \mathscr{X}_{g}(M)$ then $\Omega$ is conformally symplectic
and the conformally symplectic structure gives us the following results.

\begin{dfn}
A closed orbit $\eta$ is called a \emph{prime} orbit if it is not an iterate of a 
closed orbit of a smaller period.
\end{dfn}

\begin{prop}\label{multiplicarautovalores}
Fix $\varepsilon > 0$ and $\alpha \in \mathbb{R}$. 
Given $E \in \mathscr{X}_{g}(M)$ and the Gaussian thermostat $(M,g,E)$ 
with conformally Hamiltonian structure $(M, H, \Omega)$
such that $\gamma = \langle E,.\rangle$ and
$\Omega = \omega - \gamma \land \kappa$.
Suppose the Gaussian thermostat has a prime closed orbit $\eta$ 
and the eigenvalues of the linear Poincar\'e application in
$p \in \eta$ satisfies
$\lambda_{i}\lambda_{i+n}=e^{\beta}$.
Then it is possible to realize a surgery to obtain a Gaussian thermostat
$(M,g,\tilde{E})$ such that
$\| \tilde{E} - E\|_{C^{0}} < \frac{\alpha}{period(\eta)} + \varepsilon$,
the Gaussian thermostat $(M,g,\tilde{E})$ has $\eta$ as an orbit and 
the eigenvalues of the linear Poincar\' e application
in $p \in \eta$ satisfies
$\lambda_{i}\lambda_{i+n}=e^{\beta+\alpha}$.
\end{prop}

\begin{proof}
Consider $\mu$ Dirac function supported in $\eta$.
By the theorem [\ref{oseletes}], 
\begin{eqnarray*}
\beta &=& \int_{M} \gamma (E(x)) d\mu (x)
\end{eqnarray*}
which $\beta=\lambda_{1}\lambda_{n+1}=\lambda_{2}\lambda_{n+2}=\dots=\lambda_{n}\lambda_{2n}$.
Write $E = E_{0} \oplus E_{1}$ where
$E_{0}$ is the component of $E$ parallel to $\dot{\eta}$ and
$E_{1}$ is the orthogonal component.

Let

\begin{itemize}
\item $W \subset M$ be a tubular neighborhood of $c = \pi \circ \eta,$
\item $\tau > 0$ such that $m \tau = period ( \eta )$ with $m \in \mathbb{N}$ and
$\tau < r_{inj}$ with $r_{inj}$ is the injectivity radius of $M$.
Consider, for $ 0 \leq k < m$, $\eta_{k}(t) = \eta( (t + k)\tau)$  with $t \in [0, 1]$.
We call $c_{k}$ the projection of $\eta_{k}$ in $M$, $c_{k}= \pi \circ \eta_{k}$.
\end{itemize}

The segments $c_{k}$ can intersect transversally.
Suppose for $c_{0}$ the set of intersection points of $c_{k}$ with $c_{0}$ is 
$\mathcal{F}_{0} = \{ p_{1}, \dots, p_{l} \}$. Note that $l < m$.

Consider the sequence $t_{0}=0, \dots, t_{i}, \dots, t_{l+1}=1$
where $t_{i}$ is such that $c(t_{i})=p_{i}$ for each $p_{i} \in \mathcal{F}_{0}$.
Let $h_{i}:[0,1] \longrightarrow \mathbb{R}$ bump functions
satisfying $\int_{t_{i}}^{t_{i+1}}h_{i}(s)ds=1$ and $h_{i}(t_{i})=h_{i}(t_{i+1})=0$ for $i=0,\dots,l$.
Also consider $V \subset M$ a neighborhood of $c_{0}$ and  
a local chart $\phi(t,x)$ of $M$ such that $\phi(t,0)=c_{0}$.
Finally, let $g_{i} : V_{i} \longrightarrow \mathbb{R}$ a bumb function such that $supp(g_{i}) \subset V$
and $g_{i}(t,0)=1$.

We locally change de vector field $E$ as $\tilde{E}(\phi(t,x)) = \frac{\alpha}{m(l+1)} h_{i}(t) g_{i}(t,x) \frac{E_{0}(\phi(t,x))}{\| E_{0} \|} + E(\phi(t,x))$ and repeat this procedure to $c_{k}$, $k=1,\dots, m$.
Consider the conformally Hamiltonian flow defined by  $(M, H, \tilde{\omega})$
such that $\tilde{\gamma} = \langle \tilde{E},. \rangle$ and
$\tilde{\Omega} = \omega + \tilde{\gamma} \land \kappa$.

Thus $\eta$ is a periodic orbit associated to the flow $\tilde{\phi}$ and, moreover,
\begin{eqnarray*}
\tilde{b} &=& \int_{M^{c}} \tilde{\gamma} (F(x)) d\mu (x)\\
          &=& \int_{M^{c}} \langle \tilde{E}(\pi(x)), F(x) \rangle d\mu (x)\\
          &=& \sum_{k=1}^{m} \int_{0}^{\tau} \langle \tilde{E}(\pi (\eta(t)), F(\pi (\eta(t))) \rangle dt \\
          &=& \beta+\alpha
\end{eqnarray*}
\end{proof}

\begin{prop}\label{multiplicarautovalores2}
Fix $\varepsilon > 0$.
Let $E \in \mathscr{X}_{g}(M)$ and the Gaussian thermostat $(M,g,E)$ 
with conformally symplectic Hamiltonian structure $(M, H, \Omega)$
such that $\gamma = \langle E,.\rangle$ and
$\Omega = \omega - \gamma \land \kappa$.
Suppose that the Gaussian thermostat has a closed orbit $\eta$ with an eigenvalue of the 
linear Poincar\' e application in $p \in \eta$ satisfying $$\lambda_{i}\lambda_{i+n}=e^{\beta}.$$
Then there exists a surgery $(M,g,\tilde{E})$ of $(M,g,E)$ such that 
$\| \tilde{E} - E\|_{C^{\infty}} < \varepsilon$, 
the Gaussian thermostat $(M,g,\tilde{E})$ has $\eta$ as an orbit and, furthermore, the 
eigenvalue of the linear Poincar\'e application in $p \in \eta$ satisfies
$$\lambda_{i}\lambda_{i+n}>e^{\beta}.$$
\end{prop}

\begin{proof}
It is enough to consider the last proof with $h$ of class $C^{\infty}$ such that $h(p_{1})=0$, 
$h(x)\neq 0$ for $x \in [0,\tau]\setminus \{ p_{i}\}_{i=1,\dots, l}$ and
$\| h \|_{C^{\infty}} < \epsilon$. %verificar depois
\end{proof}

Now, we introduce local coordinates useful to this work.

\subsection{Fermi coordinates and Jacobi equation}

Given $(M,g)$ a Riemannian manifold of dimension $n+1$ and $E \in \mathscr{X}^{r}(M)$,
we denote
\begin{itemize}
\item[(i)] $\pi : SM \to M$ the unitary bundle,
\item[(ii)] $\phi^{t} : SM \to SM$ the Gaussian thermostat flow $(M, g, E),$
\item[(iii)] $\exp_{\theta} (v)$ 
the exponential application in $v \in T_{\theta}SM$ at $\theta \in SM$.
\end{itemize}

Let $\eta : [0, \tau] \to SM$ be an orbit segment and consider $c(t) = \pi \circ \eta (t)$ the
projection of $\eta$ in $M$.
We assume $\tau < r_{inj}$, where $r_{inj}$ is the injectivity radius of $(M, g)$.
Let $\{ e_{0}= c'(0), ~e_{1}, ~\dots, ~e_{n-1}\}$ be a basis of $T_{c(0)}M$ and
consider $e_{i}(t)$, $i=1,\dots, n-1$, the parallel transport of $e_{i}$ along $c$
with respect to the Riemannian connexion $\nabla$.

Consider $\psi : [0,\tau]\times\mathbb{R}^{n} \to M$ given by

$$
\psi (t,x)  = exp _{c(t)} \sum_{i=1}^{n-1} x_{i}e_{i}(t)
$$

This application is a diffeomorphism on a neighborhood $V \subset [0,\tau] \times \{0, \dots, 0 \}$
then it defines a local coordinated system in a neighborhood of $\eta$.
Furthermore, $c(t) = exp_{c(t)}0 = \psi (t, 0, \dots, 0)$ and
considering the canonical parametrization  $\tilde{\psi}$ de $TM$,
we have
$$\eta (t) = (c(t), \dot{c} (t)) =  \tilde{\psi} (~t, ~0, ~\dots, ~0, ~1, ~0, ~\dots, ~0).$$

Using the identity
$\frac{d}{dt} (d\phi_{t}) = (dX \circ \phi_{t}) d\phi_{t} $ with $X = \frac{d}{dt}\phi_{t}$,
we obtain the differential equation to the linearization of the conformally Hamiltonian flow 
over the orbit $\eta(t)$, which we call \emph{Jacobi equation} in $TM$ and are summarize in the next theorem:

\begin{teo} \label{eqjacobicdt}
Let $\eta: (-\varepsilon, \varepsilon) \longrightarrow SM$ a Gaussian thermostat orbit segment. Then the linear Poincar\' e cocycle $\hat T$ along $\eta$ satisfies the \emph{Jacobi equation}
which is written in Fermi coordinates as 
\begin{eqnarray*}
\frac{d}{dt}
\hat{T}
\Bigg{|}_{(t, x=0)}
=
\left\{
\left[\begin{array}{cc}
 0 &  I\\
 \hat{K} &  0
\end{array}\right]
+
\left[\begin{array}{cc}
  0  & 0   \\
  \hat{B}  & \hat{C}
\end{array}\right]
\right\}
\hat{T},
\end{eqnarray*}
where 
\begin{eqnarray*}
\hat{K} &=& (-K_{ij})_{ij} \quad i,j \in \{ 1, \dots, n \},\\
\hat{B} &=& \Big{(} \frac{\partial E_{i}}{\partial x_{j}}\Big{)}_{ij} \quad i,j \in \{ 1, \dots, n \},\\
\hat{C} &=& - \frac{E_{0}}{c} I = \sigma I \quad \text{with $I$ the $n \times n$ identity}.\\
\end{eqnarray*}
\end{teo}

\section{Proof of theorem A}

In this section, we show that the set of vector fields 
which define a Kupka-Smale Gaussian thermostat
is generic in $\mathscr{X}_{g}(M)$.
Before rigorously stating the theorem, 
we need the following definitions.

\begin{dfn}
We say a Gaussian thermostat $(M,g,E)$ over a compact manifold $M$ 
is \emph{Kupka-Smale} if it satisfies:
\begin{itemize}
\item[(i)] The closed orbits are hyperbolic,
\item[(ii)] The heteroclinic intersections are transversal.
\end{itemize}
\end{dfn}

%\begin{dfn}
%A property $P$ is generic in $\mathscr{X}_{g}(M)$ if there exists a residual subset
%in $\mathscr{X}_{g}(M)$ which satisfies $P$.
%\end{dfn}

\begin{teoa}
The Kupka-Smale property is generic in $\mathscr{X}_{g}(M)$.
\end{teoa}

The vector field $E$ defines the $1$-form $\gamma_{E}( \cdot )= \langle E, \cdot \rangle$.
Consider $Per(E)$ as the set of periodic orbits of $(M,g,E)$ and  $\eta \in Per(E)$ 
a prime periodic orbit with period $L$. We define the application 
$\beta_{E} : Per(E)\longrightarrow \mathbb{R}$ 
by $\beta_{E} (\eta) = \int_{0}^{L}\gamma_{E}(\dot{\eta}(s))ds$.
We can also denote $\beta_{E}(\theta)$
by $\beta_{E}(\eta)$ where $\theta$ is a point of $\eta$.
We split the main theorem in two parts, i.e., in two lemmas:

\begin{dfn}
The subset $G_{1} \subset \mathscr{X}_{g}(M)$ of vector fields 
which defines Gaussian thermostats such that if  $\eta$ is a periodic orbit then
\begin{itemize}
\item[(i)] the transversal cocycle associated  is hyperbolic, 
\item[(ii)]  $\beta_{E}(\eta) \neq 0$.
\end{itemize}
\end{dfn}

\begin{lm}\label{ksg1}
The subset $G_{1} \subset \mathscr{X}_{g}(M)$ is residual in $\mathscr{X}_{g}(M)$.
\end{lm}

\begin{dfn}
Let $G_{2}  \subset G_{1}$
such that if  $\eta_{i}$, $\eta_{j} \in Per(E)$ and
$W^{u}(\eta_{i}) \cap W^{s}(\eta_{j}) \neq \emptyset$ then this intersection is transversal.
\end{dfn}

\begin{lm}\label{ksg2}
The set $G_{2}$ is residual in $\mathscr{X}_{g}(M)$.
\end{lm}

Fixed a coordinate system along the orbit $\eta$ such that for $\tau >0$
the segment $\eta([0,\tau])$ is contained on the image of this 
coordinated system, we consider
the application  $\psi : [0,\tau] \times \mathbb{R}^{n} \longrightarrow \mathbb{R}$ 
of class $C^{\infty}$ such that $\int_{0}^{\tau}\psi(t, 0)dt = 1$
and $\psi$ has support on a neighborhood of $[0,\tau]\times \{ 0 \}$ 
contained in the image of that coordinate system.  
Furthermore we consider,
$\pi_{V(\theta)} : T_{\theta}SM \longrightarrow H(\theta)$  the projection on the vertical space,
and $\pi_{H(\theta)} : T_{\theta}SM \longrightarrow V(\theta)$ 
 the projection on the horizontal space.

\subsection{Auxiliary results}

We  state a transversality theorem due to  Abraham. 
Let
$\mathcal{A}$ be a Baire topological space, 
$M$ and $N$
manifolds satisfying the second axiom of enumerability
and with finite dimension, 
$K\subset M$ a subset,
$V\subset N$ a submanifold and an application   
\begin{eqnarray*}
F: \mathcal{A} &\longrightarrow& C^{1}(M,N)\\
a&\longmapsto& F_{a}.
\end{eqnarray*}
If the map 
\begin{eqnarray*}
ev_{F}: \mathcal{A}\times TM &\longrightarrow & TN\\
(a,v) &\longmapsto & DF_{a}v
\end{eqnarray*}
is continuous, we call $F$ a $C^{1}$-pseudorepresentation.

If $F$ is a $C^{1}$-pseudorepresentation and there exists a dense subset
$D \subset A$ such that for $a\in D$ there exists an open set $B_{a}$
in a separable Banach space, $\psi_{a}:B_{a}\longrightarrow \mathcal{A}$
continuous and $a^{\prime} \in B_{a}$ such that
\begin{itemize}
\item[(i)]
$\psi_{a}(a^{\prime})=a,$
\item[(ii)]$ev_{F\psi_{a}}:B_{a}\times M \longrightarrow N,$
is $C^{r}$ transversal to $V$ in $a^{\prime}\times K,$
\end{itemize}
then we say $F$ is $C^{r}$-pseudotransversal to $V$ in $K$.

\begin{teo}[Abraham' s Transversality theorem]\label{abraham}
Suppose $F:\mathcal{A}\longrightarrow C^{1}(M,N)$
is $C^{r}$-pseudo transversal to $V$ in $K$ with
$$
r\geq max(1,1+dimM-codimV).
$$
Let $R=\{ a\in \mathcal{A} : F(a) \pitchfork_{K} V\} = \{a \in \mathcal{A}: F(a)
\text{ is transversal to } V \text{ in points of } K\}$.

If $K=M$ then $R$ is residual in $A$. If $V$ is a closed submanifold and $K\subset M$
is compact then $R$ is open and dense in $A$.
\end{teo} 

\begin{lm}\label{lmcontinuacao}
Let $T$ be a positive number  and  $\eta$ be a hyperbolic periodic orbit of the Gaussian thermostat $(M, g,E)$ with 
$E \in \mathscr{X}_{g}(M)$ and such that $\beta_{E}(\eta)\neq 0$. 
Then there exist neighborhoods $U \subset SM$ of $\eta$ and 
$\mathcal{U} \subset \mathscr{X}_{g}(M)$ of $E$ such that
\begin{itemize}
\item[(i)] all $\tilde{E} \in \mathcal{U}$ has a periodic orbit 
$\eta_{\tilde{E}}\subset U$ and all orbit of $(M, g,\tilde{E}),$
different from $\eta_{\tilde{E}}$ passing by $U$ has period $>T,$ 
\item[(ii)] the orbit $\eta_{\tilde{E}}$ depends continuously on $\tilde{E}$,
\item[(iii)] $\beta_{\eta_{\tilde{E}}}(\eta_{\tilde{E}})\neq 0.$
\end{itemize}
\end{lm}

\begin{proof}
Let $\Sigma$ be a transversal section over $\theta \in \eta$
and consider $P_{E}:\Sigma \longrightarrow \Sigma$ the Poincar\'e application  
associated to $E$ in $\theta$. Let $L$ be the period of $\eta$ and $n$ 
a positive integer
such that $nL > 2T$. The Poincar\'e application depends 
continuously of the Gaussian thermostat and therefore, for small enough $V \subset \Sigma$, 
the application $(P_{\tilde{E}})^{n}$ is defined in $V$ for all 
$\tilde{E}\in \mathcal{U}$.

The point $\theta$ is a hyperbolic fixed point for $P_{E}$,
then there exists for possibly smaller $\mathcal{U}$ and $V$,
a continuous application $\rho:\mathcal{U}\longrightarrow V$
which associate each $E\in \mathcal{U}$ to the unique fixed point 
$\rho(E)$ of $P_{\tilde{E}}$ 
in $V$ and $\rho(E)$ is a hyperbolic point. 

By the Hartman-Grobman theorem and continuous dependency of the Poincar\'e application, 
there exist a neighborhood $\tilde{V}\subset V$ of $\theta$ and a neighborhood  
$\mathcal{U}$ of $E$ such that for $\tilde{E}\in \mathcal{U}$
then $(P_{\tilde{E}})^{k}(\theta)\in V$, $k=1,\dots,n$, and every closed orbit
of $\tilde{E}\in \mathcal{U}$ 
different from $\eta_{\tilde{E}}$ has period larger than $T$.

Furthermore, for possibly smaller $\tilde{V}$, we have
$\beta_{\eta_{\tilde{E}}}(\eta_{\tilde{E}})\neq 0$.
To finish the proof, let $U=\cup_{t\in [0,L+\varepsilon]}\phi^{t}(\tilde{V})$ for $\varepsilon$
small enough.
\end{proof}

\begin{lm}\label{periodomaiorquet}
Let $E \in TG(M,g)$ and $K \subset SM$ be a compact subset
such that all closed orbits of $K$ have period larger than  $ T$. Then there exists a 
neighborhood $\mathcal{U} \subset \mathscr{X}_{g}(M)$ of $E$ such that if
$\tilde{E} \in \mathcal{U}$
then the closed orbits of $\tilde{E}$ passing by $K$
have period larger than  $T$. 
\end{lm}

\begin{proof}
Let $\theta \in K$. The orbit $\theta$ is regular or has period $>T$ then there exists
$\varepsilon>0$ and a neighborhood $U_{\theta}$ of $\theta$ such that for 
$\tilde{\theta}\in U_{\theta}$ we have $\phi^{t}(\tilde{\theta})\notin U_{\theta}$
for $t \in [\varepsilon, T+\varepsilon]$. 

Furthermore, the flow depends continuously on $E$ and there exists a neighborhood 
$\mathcal{U}_{\theta} \subset \mathscr{X}_{g}(M)$ such that the same property holds 
for all Gaussian thermostats in $\mathcal{U}_{\theta}$. 

Consider an open cover of $K$ formed by $U_{\theta}$,  
$\{ U_{\theta}\}_{\theta \in K}$, and let a finite subcover 
$\{ U_{\theta_{l}}\}_{l=1,\dots,k}$ . The neighborhood 
$\mathcal{U} \subset \mathscr{X}_{g}(M)$
which satisfies the properties of the lemma is
$\mathcal{U}=\cap_{l=1}^{k}U_{\theta_{k}}$.
\end{proof}

The next two results are about  isotropic subspaces.

\begin{dfn}
Let $\theta \in SM$ and a subspace $U \subset T_{\theta}SM$. 
We call $U$ isotropic 
when $\Omega(u,v)=0$ for $u,v \in U$.
\end{dfn}

\begin{lm}\label{isotropico}
Let $\theta \in \hat{T}_{\theta}SM$ and an isotropic subspace  
$Q \in S(\theta)=\hat{T}_{\theta}SM$ then
the set $t \in \mathbb{R}$ such that 
$D\phi_{t}Q \cap V(\phi_{t}(\theta ))\neq \{ 0 \}$ is discrete.
\end{lm}

\begin{proof}
First let us show  show if $Q$ is an isotropic subspace of $T_{\theta}SM$
such that $E\cap V(\theta ) \neq \{ 0 \}$ then there exists a neighborhood
$W$ of $t=0$ such that $D\phi_{t}Q \cap V(\phi^{t}(\theta) )=\{ 0 \}$ 
for all $t \in W\setminus \{ 0 \}$.

Let $\pi_{H(\theta)} : \hat{T}_{\theta}SM \longrightarrow H(\theta)$ be  the orthogonal projection
over the horizontal bundle. There exists a isometry
$J_{\theta}: \hat{T}_{\theta}SM \longrightarrow \hat{T}_{\theta}SM$ 
such that $J_{\theta}^{2} = - I$, $J_{\theta}V(\theta)=H(\theta)$,
$J_{\theta}H(\theta)=V(\theta)$.
The subspaces $\pi_{H}(Q)$ and $J_{\theta}(Q\cap V(\theta))$ are orthogonal. 
In fact, let $x  \in \pi_{H}(Q)$. We can write $x=y-z$ where
$y \in Q$ and $z \in Q \cap V(\theta)$. If $Jw \in J_{\theta}(Q\cap V(\theta))$
then $\Omega(y,w)=0$ because $Q$ is isotropic and, therefore, $(\gamma \land \kappa) (y,w) =0$
because $w$ is a vertical vector. So $g(y,Jw)=0$. We also have $g(z,J_{\theta}w)=0$ as $z, w \in Q\cap V(\theta)$.Then
$g(x,Jw)=0$.

Let $\{h_{1},\dots, h_{k} \}$ be a basis of $\pi_{H(\theta)}(Q)$.
If $t$ is sufficiently small then there exists a set of linearly independent vectors
$\{h_{1}(t),\dots, h_{k}(t) \} \subset \pi_{H(\phi^{t}(\theta))}D\phi^{t}Q$.

Let $\{ w_{1},\dots, w_{l} \}$ be a basis of $Q\cap V(\theta)$ and 
 the Jacobi fields  $J_{1}, \dots, J_{l}$ with $J(0)=0$ and 
$\dot{J}(0)=\pi_{V(\theta)}w_{i}$,
for $i=1,\dots, l$. Define, for $t>0$, the vector fields 
$W_{i}(t)=-\frac{1}{\| J_{i}(t) \|}J_{i}(t)$ and thus
$$
\lim_{t \to 0}W_{i}(t)=\lim_{t \to 0} \frac{J_{i}(t)}{\| J_{i}(t) \|}=-\pi_{V(\theta)}w_{i}=Jw_{i}.
$$
So $\pi_{H(\phi^{t}(\theta))D\phi_{t} w_{i}}=J_{i}(t) \neq 0$ when $t>0$. Thus $\{ h_{1}(t),\dots,h_{k}(t), w_{1}(t),\dots, w_{l}(t)\}$  is included in $\pi_{H(\phi^{t}(\theta))}(D\phi_{t}Q)$.
When $t$ is small this set is close to 
$\{ h_{1},\dots,h_{k}, J_{\theta}w_{1},\dots, J_{\theta}w_{l}\}$
then it is a linearly independent set. 

Finally, we conclude for $t$ sufficiently small  
$dim(\pi_{H(\phi^{t}(\theta))}D\phi^{t}Q) = dim(Q)$ and then
$D\phi^{t}Q \cap V(\phi^{t}(\theta)) = \{ 0 \}$.   
\end{proof}

\begin{cor}
Let $\theta \in SM$ and $Q \in \hat{T}_{\theta}SM$ be an isotropic subspace. Then the set
$t \in \mathbb{R}$ such that $dim(\pi_{H}(D\phi^{t}Q)) \neq dim(Q)$  
is discrete.
\end{cor}
After previous corolary, we say that  isotropic spaces of Gaussian thermostat deviate from the vertical.

Finally, a technical lemma about symmetric applications.

\begin{lm}\label{algebralinear}
Let $U$ 
and $Q$ be  subspaces of $\mathbb{R}^{n}$ such that $dim (U) = dim (Q)$.
Then there exists a symmetric application 
$\tilde{B}: \mathbb{R}^{n}\longrightarrow \mathbb{R}^{n}$ such that
$$\pi_{Q}(U) = Q$$
where
$\pi_{Q}$ is the orthogonal projection over $Q$. 
\end{lm}

\begin{proof}
Consider the following decomposition of $\mathbb{R}^{n}$:
$$
U \cap Q \oplus U \cap Q^{\perp} \oplus U^{\perp} \cap Q \oplus U^{\perp} \cap Q^{\perp}. 
$$

Suppose $dim(U \cap Q)=l$ and take $\{w_{1}\dots, w_{l}\}$
a basis of $dim(U \cap Q)$. 
Define on this subspaces $\tilde{B}(w_{i})=w_{i}$, for $i=1,\dots, l$,
the identity application.

Suppose $dim(U \cap Q^{\perp})=k$ then $dim(U^{\perp} \cap Q)=k$.
Let $\{ u_{1} \dots, u_{k}\}$ be a basis of $dim(U \cap Q^{\perp})$ and
$\{ v_{1}, \dots, v_{k}\}$ a basis of $U^{\perp} \cap Q$.
Define $\tilde{B}(u_{i})=v_{i}$, for $i=1,\dots, k$. 
On the subspace $U^{\perp} \cap Q$
define the application $\tilde{B}$ to make it symmetric.

Suppose $dim(U^{\perp} \cap Q^{\perp})=m$ and $\{ z_{1},\dots, z_{m} \}$
is a basis of $U^{\perp} \cap Q^{\perp}$. Define $\tilde{B}(z_{i})=0$, for $i=1,\dots, m$. 
So we have that the application $\tilde{B}$ is completely defined.
\end{proof}

\subsection{Genericity of hyperbolic closed orbits}

In this section we prove lemma \ref{ksg1}.
Let
\begin{itemize} 
\item $\eta$ be a periodic orbit of period $L$ of the Gaussian thermostat
$(M, g, E)$ with $E \in \mathscr{X}_{g}(M)$,
\item $\theta \in \eta$ be a point of this orbit,
\item $\Sigma$ be a transversal section in $\theta \in \eta$, 
\item $T: \hat{T}_{\theta}SM \longrightarrow \hat{T}_{\theta}SM$ 
be the transversal derivative cocycle associated to $\theta$.
\end{itemize}
By the implicit function theorem, there exist $\mathcal{O} \subset \mathscr{X}_{g}(M)$
a neighborhood of $E$ and
$U \subset SM$ a neighborhood of $\theta$ 
such that the Poincar\'e application 
$$P: \mathcal{O} \times \Sigma \cap U \longrightarrow \Sigma,$$
is well defined as their iterates. 
Let
\begin{itemize}
\item For $i=1,\dots, k$
\begin{eqnarray*}
\rho_{i}:\mathcal{O} &\longrightarrow & C^{r}(\Sigma \cap U,\Sigma\times\Sigma ),\\
 E &\longmapsto & \theta \mapsto (\theta, P^{i}_{E}\theta ),\\
\end{eqnarray*}
\item $W=\{(\theta,\theta): \theta \in \Sigma \}$,
\item $V\subset \overline{V} \subset U$ a neighborhood of $\theta$ with $\overline{V}$
compact,
\item $R_{0}=\mathcal{O}$, $R_{j}=\{ E \in \mathcal{O}: \rho_{i} \pitchfork_{\Sigma\cap \overline{V}} W \text{ para } i=1,\dots, j\}$,  
\item $S_{0}=\mathcal{O}$,
$S_{j} = \{ E \in \mathcal{O} : \text{fixed points of  } \rho_{i} \text{ are hyperbolic and }
\text{ for a fixed point } \theta \text{ of } \rho_{i},\\ \beta_{E}(\theta) \neq 0 \text{ to } i=1,\dots,j\}$.
\end{itemize}

From  theorem \ref{eqjacobicdt}, we can fix Fermi coordinates
along $\eta$ and the application $T$ 
is written as $T=e^{\mathbb{A}}$ with $\mathbb{A}$ 
of the form

$$
\mathbb{A}= \begin{bmatrix}0 & I \\ S & \lambda I
\end{bmatrix} 
$$

wherein
\begin{itemize}
\item $S=\int_{0}^{L}\hat{K}(t)+\hat{B}(t)dt$ is a symmetric matrix such that
\begin{eqnarray*}
\hat{K}(t) &=& (-K_{ij})_{ij}(t) \quad i,j \in \{ 1, \dots, n \},\\
\hat{B}(t) &=& \Big{(} \frac{\partial E_{i}}{\partial x_{j}}\Big{)}_{ij} (t)\quad i,j \in \{ 1, \dots, n \},\\
\end{eqnarray*}
\item $\lambda = - \int_{0}^{L}E_{0} (t) dt= - \int_{0}^{L} \gamma_{E}(v(t)) dt= - \int_{0}^{L} \langle E, v \rangle (t) dt $ ( at the level $c=1$ ), 
\item  $I$ is the identity matrix.
\end{itemize}

\begin{lm}\label{lmhiperbolico}
If $\rho_{i}(\theta) \in W$, $\rho_{j}(\theta) \notin W$, $j < i$ and
$E \in \mathscr{X}_{g}(M)$ then there exists 
a $\varepsilon$-perturbation $\tilde{E}\in \mathscr{X}_{g}(M)$ of $E$ 
such that $\theta$ is a hyperbolic periodic orbit and $\beta(\theta) \neq 0$.
\end{lm}

\begin{proof}
Lets make a $\varepsilon$-perturbation $\tilde{E}^{1}$ of the Gaussian thermostat
$E$ such that $\tilde{E}^{1}$ preserves  the periodic orbit
$\eta$ and the eigenvalues of $T$ have modulus different from $1$. 

In Fermi coordinates, the matrix $T_{\theta}=e^{\mathbb{A}}$ 
has an eigenvalue of modulus $1$ if and only if the matrix
$\mathbb{A}$ has an eigenvalue with real part equals to zero. 
If the eigenvalue is complex, we apply proposition
4.4 
with $\alpha=\varepsilon > 0$.

In addition, 
$det (\mathbb{A})= (-1)^{n} det(S)$ 
is equal to the product of its eigenvalues. 
Thus the application $\mathbb{A}$ has real eigenvalues with modulus
equal to $0$ if and only if $det(S)=0$. 
We denote $S^{k}$, $k\in\{1,\dots,n\}$ the matrix $(n-k)\times (n-k) $ 
constructed from $S$
removing the first $k$ lines and $k$ columns.

Let $k_{0}$ the smaller $k$ such that $det(S^{k})\neq 0$. If $det(S^{k}) = 0$
for $k=1,\dots,n$ take $k_{0}=n$.

Consider $\tau > 0$ such that $\eta([0,\tau])$
is contained in a coordinate Fermi neighborhood.

Take $\lambda_{1}<\varepsilon \tau$. 
The perturbation on the Gaussian thermostat is constructed from a perturbation 
$\tilde{E}^{1}$ of the vector field $E$:
$$
\left\{\begin{array}{rcl} 
\tilde{E}_{0}^{1}(t,x) &=&  E_{0}\\
\tilde{E}_{1}^{1}(t,x) &=& E_{1} + \psi(t,x)\frac{\lambda_{1}}{\tau}x_{1} \\
& \dots &\\
\tilde{E}_{k_{0}}^{1}(t,x) &=& E_{k_{0}} + \psi(t,x)\frac{\lambda_{1}}{\tau}x_{k_{0}}\\
\tilde{E}_{k_{0}+1}^{1}(t,x) &=&  E_{k_{0}+1}\\
&\dots &\\
\tilde{E}_{n}^{1}(t,x) &=& E_{n}\\
\end{array}\right.
$$
We can write the perturbation matrix $\tilde{S}$ as
$$\tilde{S} = \hat{K} + \hat{B} + \varepsilon D = S + \varepsilon D, $$ 
wherein $D$ is a diagonal matrix such that if $i \leq k_{0}$ then $D_{ii}=1$ and
if $i > k_{0}$ then $D_{ii}=0$.

Calculating the determinant of $\tilde{S}$ by cofactors in terms of the first row, we have
\begin{itemize}
\item If $k_{0} < n$ then
$\det (\tilde{S}) 
> \lambda_{1}^{k_{0}} det (S^{k_{0}}) \neq 0,$

\item If $k_{0} = n$ then $det(\tilde{S}) > \lambda_{1}^{n} \neq 0$.
\end{itemize}

If $\beta_{\tilde{E}^{1}}(\eta) \neq 0$ then there is nothing to do. Otherwise,
take $0 <\lambda_{2} < \varepsilon \tau$ and $\tilde{E}^{2}$ defined by
$$
\left\{\begin{array}{rcl} 
\tilde{E}_{0}^{2} &=& \tilde{E}_{0}^{1} + \psi(t,x)\frac{\lambda_{2}}{\tau}\\\
\tilde{E}_{i}^{2} &=& \tilde{E}_{i}^{1} \text{ para } i=1,\dots, (n-1)
\end{array}\right.
$$

Consider the application $\beta_{\tilde{E}^{2}}$ 
applied to $\eta$ 
which is also an orbit of the Gaussian thermostat $(M, g, \tilde{E}^{2})$, then

$$
\beta_{\tilde{E}^{2}}(\eta) = \int_{0}^{\tau}\gamma_{\tilde{E}^{2}}(\dot{\eta}(s))ds = \lambda_{2} > 0.
$$
\end{proof}

\begin{lm}\label{lmcrtransverso}
Suppose $S_{j-1}$ is open and dense in $\mathcal{O}$. 
Then $\rho_{j}$ is a pseudo $C^{r}$-representation transversal to $W$ 
in $\Sigma \cap \overline{V}$. 
\end{lm}

\begin{proof} With $B=TG(M,g)$ we denote the Banach space of perturbations and with  $S_{j-1}\cap TG(M,g)$ the dense set provided from the definition of pseudo transversality. We need to prove 
$\rho_{j}$ is $C^{r}$-transversal to $W$ in $E \times \Sigma \cap V$
for $E\in S_{j-1}$.

If $\rho_{i}(E,\theta) \in W$ for $i<j$ and $\rho_{j}(E,\theta) \in W$ then $\theta$ 
is a periodic hyperbolic point, $\rho_{j}(E)\pitchfork_{\theta}W$, and $\rho_{j}\pitchfork_{(E,\theta)}W$.

If $\rho_{i}(E,\theta) \notin W$ for $i<j$, $\rho_{j}(E,\theta) \in W$
and $D\rho_{j}(E,\theta) \cap T_{\theta}W = \{ 0 \}$ then we have the transversality.

If  $D\rho_{j}(E,\theta) \cap TW \neq \{ 0 \}$ then
let $u \in T_{\theta}SM$ such that $D\rho_{i}(u)=(u,D\phi^{i} u=u)$.
Moreover, the projection of an eigenvalue of $D\phi^{i}$ in $\hat{T}_{\theta}SM$ 
is an eigenvalue of the linear Poincar\'e application.

We can take $u$ with nonzero horizontal component.
In fact, by lemma \ref{isotropico}, the subspace generated by $u$, denoted by $U$,
is isotropic and then the intersection with the vertical bundle occurs in isolated points
along the orbit of $\theta$. It is enough to consider the point $\theta^{\prime}$ 
in $\eta \in \theta$ such that $U \cap V(\theta^{\prime})=\{ 0 \}$.
We  keep calling $\theta^{\prime}$ of $\theta$.

We consider $Q \subset T_{\theta}SM$ the orthogonal subspace of $u$
Fix Fermi coordinates along the orbit segment
$\eta:[0,\tau]\longrightarrow SM$ ($\tau < r_{inj}$) in such way that
$\eta(0)=\tilde{\theta}$ and $\eta(\tau)=\theta$.

Consider the path $c:(-\varepsilon,\varepsilon)\longrightarrow CS(\mathbb{R}^{2n})$
such that $c(s)=e^{\mathbb{A}+sB}$ wherein $B=\begin{bmatrix}0&0\\\tilde{B}&0\end{bmatrix}$
and $\tilde{B}:\mathbb{R}^{n}\longrightarrow\mathbb{R}^{n}$ is symmetric. 

Note that $c(0)=d\phi^{\tau}_{\tilde{\theta}}$
and this path is the 
derivative cocycle of the Gaussian thermostat $(M, g,\tilde{E}^{s})$ 
wherein $\tilde{E}^{s}$ is given by

$$
\left\{\begin{array}{rcl} 
\tilde{E}_{0}^{s}(t,x) &=& E \\
\tilde{E}_{i}^{s}(t,x) &=& E_{i} + \frac{s}{\tau} 
\sum_{j=1}^{n} \psi(t,x)\tilde{B}_{ij}x_{j} \text{ for } i=1,\dots, n
\end{array}\right.
$$
We are looking for an application $B$ such that for $s \neq 0$ we have
$$
\pi_{Q}e^{i(\mathbb{A}+sB)}U \neq 0.
$$
Calculating the differential of the path in $s=0$ 
$$
\frac{d}{ds}\pi_{Q}e^{i(\mathbb{A}+sB)}U\Big{|}_{s=0} = \pi_{Q}(ie^{i\mathbb{A}})BU  
$$
this differential must be non trivial and, for this, it is enough to show $BU \neq U$ 
and $dim(BU)=dim(U)$ once $U$ is an eigenspace of the isomorphism $e^{\mathbb{A}}$.

Let $Q^{\prime} \subset Q$ 
such that $Q^{\prime} \subset V(\theta)$. This choice is possible because 
$U \cap V(\theta) = \{ 0 \}$.
Let $\tilde{B}$ be the symmetric application such that $\tilde{B}\pi_{H(\theta)}U=Q^{\prime}$
whose existence is guaranteed by lemma \ref{algebralinear}. Moreover,
due to $U \cap V(\theta) = \{ 0 \}$, we have $1=dim(Q^{\prime}=BU)=dim(U)$.
\end{proof}

\begin{lm}\label{aplicacaopoincare}
There exists a $\varepsilon$-perturbation $\tilde{E}\in \mathscr{X}_{g}(M)$ of $E$ such that
the Poincar\'e application $P$ of the Gaussian thermostat
$(M,g,\tilde{E})$ has only a finite number of fixed points, all hyperbolic, and if
$\eta$ is an orbit of the flow associated to a fixed point then $\beta_{\tilde{E}}(\eta) \neq 0$. 
\end{lm}

\begin{proof} The proof is done performing an inductive argument: By lemma \ref{lmcrtransverso},
if $S_{j-1}$ is open and dense in $\mathcal{O}$ then 
$\rho_{j}$ is a  pseudo $C^{r}$-representation transversal to $W$ 
in $\Sigma\cap \overline{V}$. 
By theorem \ref{abraham}, $R_{j}$ is open and dense in $\mathcal{O}$.
By lemma \ref{lmhiperbolico}, 
if $\rho_{i}(\theta) \in W$ and $\rho_{j}(\theta)\notin W$, $j\leq i$, then there exists
a $\varepsilon$-perturbation $\tilde{E}$ of $E$ such that $\theta$
is a hyperbolic periodic orbit.
Therefore if $R_{j}$ is open and dense $\mathcal{O}$
then $S_{j}$ is open and dense in $\mathcal{O}$.
\end{proof}

We are going to prove lemma \ref{ksg1}: 

\noindent{\em Proof of lemma \ref{ksg1}.}
We  show that given an integer $T>0$ then the set of vector fields 
$G_{1}(T) \subset \mathscr{X}_{g}(M)$ which defines Gaussian thermostats
whose orbits of period smaller than $ T$ are hyperbolic is an open and dense set.
Once $G_{1} = \cap_{T \geq 1} G_{1}(T)$ then $G_{1}$ is residual.

\begin{description}
\item[$G_{1}(T)$ is open in $\mathscr{X}_{g}(M)$.]
Let $E \in G_{1}(M)$. The Gaussian thermostat $(M, g, E)$ has only a finite number of 
orbits of period $\leq T$. 

Let $\theta \in SM$. We have two possibilities:
\begin{itemize}
\item[(i)] 
The point $\theta$ is contained in a regular orbit or a periodic orbit of period $> T$.
By the tubular flow theorem, there exists a neighborhood $U_{\theta}$ of $\theta$ in $SM$
such that every orbit of $(M, g, E)$ which intersects $\overline{U}_{\theta}$ has period $>T$.
By lemma \ref{periodomaiorquet}, there exists a neighborhood $N_{\theta} \subset \mathscr{X}_{g}(M)$ of $E$ such that the orbit of a Gaussian thermostat in $N_{\theta}$ 
which passes through $U_{\theta}$ 
has periodic orbit of period $> T$. 
\item[(ii)]
The point $\theta$ is contained in a periodic orbit of period $\leq T$ e $\beta_{E}(\theta)\neq 0$.
By lemma \ref{lmcontinuacao}, there exist neighborhoods $U_{\theta} \subset SM$ 
and 
$N_{\theta} \subset \mathscr{X}_{g}(M)$ such that if $E\in N_{\theta}$ has a unique closed hyperbolic orbit  $\theta_{E}$ in $U_{\theta}$ which is hyperbolic then $\beta_{E}(\theta_{E})\neq 0$ and all others periodic orbits intersecting $U_{\theta}$
have period  $>T$.
\end{itemize}

Let $\{ U_{\theta}, \theta \in SM \}$ an open cover of $SM$. 
Take a finite subcover $U_{1},\dots,U_{k}$ and consider $N_{1},\dots,N_{k}$
from this cover. 
The vector fields in $N = N_{1} \cap \dots \cap N_{k}$ 
define Gaussian thermostats with all periodic orbits $\theta_{E}$
close to periodic orbits of  $(M,g,E)$ with period $\leq T$ 
also have period $\leq T$. 
Furthermore, such orbits $\theta_{E}$ are all hyperbolic and $\beta(\theta_{E})\neq 0$.

\item[$G_{1}(T)$ is dense in $\mathscr{X}_{g}(M)$.]
Consider the set $\Gamma(T)=\{ \theta \in SM : \mathcal{O}(\theta) \text{ is closed with period} \leq T\}$. This set is compact. In fact, 
let $\theta_{n}$ be a sequence in $\Gamma(T)$ such that $\theta_{n}\longrightarrow \theta$. 
If the orbit of $\theta$  has period $\leq T$ then $\theta \in \Gamma(T)$.
If the orbit of $\theta$  is regular or closed with period greater than 
$T$ then, by the tubular flow theorem, there exists a neighborhood $U_{\theta}$ of 
$\theta$ such that a closed orbit in $\overline{U}_{\theta}$ has period greater 
than $T$. Contradiction.

We will show there exists $\tilde{E}$ 
arbitrarily close of $E$ such that $\tilde{E} \in G_{1}(T)$.

Let $\{ W_{\theta}\}_{\theta\in\Gamma}$ be an open cover of  $\Gamma$ 
where  $W_{\theta}$ is a neighborhood of the orbit $\theta$ and let
 $\{W_{\theta_{l}}\}_{l=1,\dots,k}$ be  a finite subcover. 
Consider $W=\cup_{l=1}^{k}W_{l}$ and the compact subset $K = SM \setminus W$. 
Every periodic orbit in $K$ has period greater than $T$. 
By lemma \ref{periodomaiorquet}, there exists an open set $N \subset \mathscr{X}_{g}(M)$ 
such that every closed orbit in $K$ of a Gaussian thermostat in $N$ has period greater than $T$. 
Consider the Poincar\'e application $P_{l}$ associated to $\theta_{l}$.
By lemma \ref{aplicacaopoincare}, 
there exists a $\varepsilon$-close Gaussian thermostat such that the application $P_{l}$
has only hyperbolic periodic points and for this points $\beta \neq 0$.
Applying this lemma for $l=1,\dots,k$, we conclude the result.
\end{description}
\qed

\subsection{Transversality of invariant manifolds}

To finish the proof of  theorem A, it remain to prove  lemma \ref{ksg2}.
Consider $T>0$. 
Let $E \in G_{1}(T)$ and $\eta_{1},\dots,\eta_{l}$ be the periodic orbits of $(M,g,E)$ of period
 smaller and equal than $ T$.
For each $\eta_{i}$, we take compact neighborhoods $W^{s}_{0}(\eta_{i},E)$
and $W^{u}_{0}(\eta_{i},E)$ of $\eta_{i}$ in $W^{s}(\eta_{i},E)$
and $W^{u}(\eta_{i},E)$, respectively, such that the boundaries of   
$W^{s}_{0}(\eta_{i},E)$ and $W^{u}_{0}(\eta_{i},E)$ are fundamental domains.
Let $\Sigma_{i}^{s}$ be a submanifold of $SM$ of codimension $1$
transversal to the flow direction which intersection with $W^{s}(\theta_{i},E)$
is $\partial W^{s}_{0}(\theta_{i},E)$.
For $\tilde{E}$ in a neighborhood  $N^{s}$ of $E$
where the flow is transversal to $\Sigma_{i}^{s}$, we take a neighborhood 
$W^{s}_{0}(\tilde{\eta}_{i},\tilde{E})$ such that $\tilde{\eta}_{i}$ 
is a continuation of
$\eta_{i}$ which boundary is the intersection of $W^{s}(\tilde{\eta}_{i},\tilde{E})$ with
$\Sigma_{i}^{s}$.
We perform  the same construction for  $W^{u}_{0}(\eta_{i},E)$ to obtain the set 
$N^{u}$ and set $N=N^{s} \cap N^{u}$.
For each positive integer $n$, 
define $W^{s}_{n}(\eta_{i}, E)=\phi_{-n}(W^{s}_{0}(\eta_{i}, E))$
and $W^{u}_{n}(\eta_{i},E)=\phi_{n}(W^{u}_{0}(\eta_{i}, E))$.

Let $G_{2}^{n}(T) \subset G_{1}(T)$
such that if $E \in G_{2}^{n}(T)$ then $W^{u}_{n}(\eta_{i}, E)$ is transversal to 
$W^{s}_{n}(\eta_{j}, E)$ for all closed orbits $\eta_{i}$ and $\eta_{j}$ 
of period $\leq T$ of $E$. To prove lemma \ref{ksg2}, it is enough to show
the following lemma.

\begin{lm}
Let $E \in G_{2}(T)$ and $N$ be a neighborhood of $E$
as above. Then for all $n \in \mathbb{N}$, 
the set $G_{2}^{n}(T)$ is open and dense in $N$.
\end{lm}

\begin{proof}
We denote by $G_{2}^{n,i,j}(T)$ the set of Gaussian thermostats $E \in N$ 
such that if $\eta_{1},\dots,\eta_{k}$ are hyperbolic closed orbits of $(M,g,E)$
with period smaller and equal than $ T$ then the invariant manifolds  $W^{s}_{n}(\eta_{i},E)$ and $W^{u}_{n}(\eta_{j},E)$ 
are transversal. 

The set $G_{2}^{n}$ satisfies
$$
G_{2}^{n}(T)=\cap_{i,j=1}^{k}G_{2}^{n,i,j}(T).
$$
Therefore it is enough to show that $G_{2}^{n,i,j}(T)$ is open and dense.
\begin{description}
\item[$G_{2}^{n,i,j}(T)$ is open in $\mathscr{X}_{g}(M)$.]
Let $E\in G_{2}^{n,i,j}(T)$. Once $W^{s}_{n}(\eta_{i},E)$ and 
$W^{u}_{n}(\eta_{j},E)$ are transversal and the applications 
$E\longrightarrow W^{s}_{n}(\eta_{i},E)$ and
$E\longrightarrow W^{s}_{n}(\eta_{i},E)$ are continuous,
it follows that there exists a neighborhood $N_{ij}\in N$ of $E$
such that for all $\tilde{E} \in N_{ij}$
the manifolds $W^{s}_{n}(\eta_{i},\tilde{E})$ and 
$W^{u}_{n}(\eta_{j},\tilde{E})$ 
are transversal. 
\item[$G_{2}^{n,i,j}(T)$ is dense in $\mathscr{X}_{g}(M)$.] 
Consider the compact set $K= W^{s}_{n}(\eta_{i},E) \cap W^{u}_{n}(\eta_{j},E)$.
For $\theta \in K$ we take $A_{\theta}$ a neighborhood of $\theta$ with only one connected 
component in $K$.
We consider $\{ A_{\theta_{l}} \}_{l=1,\dots ,k}$ a finite subcover of  an open cover
$\{ A_{\theta }\}_{\theta \in K}$ of $K$. 

There exists a neighborhood $\tilde{N} \in N$ such that if $\tilde{E}\in \tilde{N}$
then $W^{s}_{n}(\eta_{i},\tilde{E}) \cap W^{u}_{n}(\eta_{j},\tilde{E})
\subset \cup_{l=1,\dots , k} A_{\theta_{l}}$.
Consider $\tilde{N}^{l} \subset \tilde{N}$ whose Gaussian thermostat
$(M,g,E)$ in that set verifies that  $W^{s}_{n}(\eta_{i},E)$ is traversal to $W^{u}_{n}(\eta_{j},E)$
in $\overline{A}_{\theta_{l}}$. We  show that $\tilde{N}^{l}$ is dense in $\tilde{N}$.
Consider $TW^{s,u}W(\theta)\in\hat{T}_{\theta}SM$ the stable and unstable manifolds tangent space, respectively, restricted to the derivative cocycle on the point $\theta \in SM$.
Now, we fix Fermi coordinates along the orbit segment $\eta:[0,\tau]\longrightarrow SM$ ($\tau < r_{inj}$) such that $\eta(0)=\tilde{\theta}$ and $\eta(\tau)=\theta$.
Consider the path $c:(-\varepsilon,\varepsilon)\longrightarrow CS(\mathbb{R}^{2n})$
such that $c(s)=e^{A+sB}$ where $B=\begin{bmatrix}0&0\\\tilde{B}&0\end{bmatrix}$
and $\tilde{B}:\mathbb{R}^{n}\longrightarrow\mathbb{R}^{n}$ is symmetric. 
Note that $c(0)= T: \hat{T}_{\tilde{\theta}}SM\longrightarrow \hat{T}_{\theta}SM$
is the transversal derivative cocycle of $(M,g,E)$ 
and this path represents the  transversal derivative cocycle of the Gaussian thermostats 
$(M,g,\tilde{E})$ where $\tilde{E}$ is given by
$$
\left\{\begin{array}{rcl} 
\tilde{E}_{0}(t,x) &=& E \\
\tilde{E}_{i}(t,x) &=& E_{i} +\sum_{j=1}^{n} \psi(t,x)\tilde{B}_{ij}x_{j} \text{ to } i=1,\dots, n.
\end{array}\right.
$$
Consider $U = \hat{T}W^{s}(\theta) \cap \hat{T}W^{u}(\tilde{\theta})$,
$Q= (\hat{T}W^{s}(\theta)+\hat{T}W^{u}(\theta))^{\perp}$ and
$dim(U)=dim(Q)=k$.

Let $\pi_{Q}$ be
the orthogonal projection on the subspace $Q$. Observe that 
$$
\pi_{Q}e^{A}U = 0
$$
and to finish the proof we  look for an application $B$ such that for $s \neq 0$ we have
$$
\pi_{Q}e^{A+sB}U \neq 0.
$$
Considering
$$
\frac{d}{ds}\pi_{Q}e^{A+sB}U\Big{|}_{s=0} = \pi_{Q}e^{A}BU  
$$
we require that
$$
BU = \pi_{V(\theta)}(BU) \subset \pi_{V(\theta)}(e^{-\mathbb{A}}Q)
$$
and $dim(U) = dim(B\pi_{H(\theta)}U) \leq \pi_{V(\theta)}(e^{-\mathbb{A}}Q)$
where $\pi_{V(\theta)}$ is the orthogonal projection on the vertical subspace 
and $\pi_{H(\theta)}$ is the orthogonal projection on the horizontal subspace.
From the fact that $\beta(\eta_{i})\neq 0$ and $\beta(\eta_{j})\neq 0$
follows that $\hat{T}W^{s}(\theta)$ and $\hat{T}W^{u}(\tilde{\theta})$ are isotropic 
and the lemma \ref{isotropico} gives us $dim(\pi_{H(\theta)}U)=dim(\pi_{V(\theta)}Q)=k$.
Finally, the existence of an application $\tilde{B}$ such that $\tilde{B}(\pi_{H(\theta)}U)=Q$ 
is guaranteed by lemma \ref{algebralinear}.
\end{description}
\end{proof}

\section{Proof of theorems B and B'}

The proof of theorem B is an adaptation to the context of conformally symplectic applications of the work in \cite{bdp}. In the proof, we use the perturbative theorem from section \ref{teofranks}.

The idea of the proof goes as follows. 
Let $\Sigma$ be  a transversal section in $\theta \in \eta$ and
$f: \Sigma \longrightarrow \Sigma$ be
the Poincar\'e  application. Consider 
$\pi: T_{\theta}SM \longrightarrow \hat{T}_{\theta}SM$
the projection of $T_{\theta}SM$ over $\hat{T}_{\theta}SM$.
Suppose $\pi \circ df_{\theta}$ 
has $2n$ invariant subspaces $E_{1}, \dots, E_{2n}$
on the homoclinic class $H(\theta,f)$.
Suppose also $H(\theta,f)$ has $2n-1$ periodic  points $\theta_{1}, \dots, \theta_{2n-1}$ such that
$\theta_{i}$ has complex eigenvalue related to the subspaces $E_{i}$ and $E_{i+1}$ (we say in this case that there is a sequence of periodic orbits with concatenated complex eigenvalues).  
Then by a perturbation $\tilde{f}$ of $f$ we can make
$\pi \circ d\tilde{f}^{per(\theta_{i})}_{\theta_{i}}(E_{i})=
\pi \circ  d\tilde{f}^{per(\theta_{i})}_{\theta_{i}}(E_{i+1})$ and 
$\pi \circ  d\tilde{f}^{per(\theta_{i})}_{\theta_{i}}(E_{i+1})=
\pi \circ d\tilde{f}^{per(\theta_{i})}_{\theta_{i}}(E_{i})$
where $per(\theta_{i})$ is the period of $\theta_{i}$.
Using the notion of transitions,  it is possible  to create a point $\theta$ whose orbit 
has some iterates near to each $\theta_{i}$ and then
the differential of $f$ on $\theta$ inherits some properties of the differential of $f$ on $\theta_{i}$.
Then it mixes all the subspaces $E_{k}$ creating the sink or the source of the theorem.
However, if in a robust way the  subspaces $E_{i}$
have not a sequence of concateneted complex eigenvalues, it is impossible to create the complex eigenvalues 
and we have dominated splitting.
The perturbative theorem from section \ref{teofranks} 
permit us to show an homoclinic class has transitions and
also to translate the dynamical problem to a 
linear systems problem. The proof is explained in details in the following subsections: in subsection \ref{transl} we makes the translation from diffeomorphisms problems to linear systems language; in subsections \ref{restr}, \ref{bidi} and \ref{generic d} we show how to create either a sink or repeller by perturbation of a linear system if there are a collection of periodic orbits with concatenated complex eigenvalues; in subsection \ref{sufi} we provide a sufficient condition to get a dominated splitting and we show that if there are not  a collection of periodic orbits with complex eigenvalues then the sufficient condition holds

The entire proof of theorem B can be done in the symplectic category, so together with the work of Contreras in
\cite{contreras} it is valid the following geodesic flow version  of theorem B:

\begin{teocs}%\label{teoc}
Given $(M,g)$ Riemannian manifold with geodesic flow $\phi$ and 
$\eta$ a saddle hyperbolic periodic orbit.
For a neighborhood $V$ of $\eta$, one of this alternatives is valid:
\begin{itemize}
\item[(i)] the homoclinic class $H(\eta,\phi) \subset V$ has a dominated splitting decomposition
\item[(ii)] given a neighborhood $U$ of $H(\eta, \phi) \subset \subset V$ and $k \in \mathbb{N}$, 
there exists a geodesic flow $(M, \tilde{g})$ arbitrarily close to $(M,g)$ with $k$ islands
arbitrarily close to $\eta$ which its orbit is contained in $U$.
\end{itemize}
\end{teocs}

The proof of theorem B' follows from the thesis of theorem B and from the fact that for $C^2$ three dimensional Kupka-Smale vector fields, if the closure of saddle periodic orbits has a dominated splitting then it is hyperbolic. The statement is proved in \cite{aubin}, which extends to three dimensional flows initially proved  for surfaces diffeomorphisms in \cite{pujalssambarino1}. Observe that we can apply the result in \cite{aubin} since we already proved that generically, Gaussian thermostat are Kupka-Smale.

\subsection{Conformally symplectic linear system and transitions}
\label{transl}

In this section we make the translation from diffeomorphisms problems to linear systems language.
 
\begin{dfn}
Let $\Sigma$ be a topological space, $f$ an homeomorphism defined in $\Sigma$,
a locally trivial fiber bundle $\pi :\mathcal{E} \longrightarrow \Sigma$ over $\Sigma$ such that
the fiber dimension $\mathcal{E}_{x}$ for $x \in \Sigma$ is constant over $\Sigma$, 
and an application $A: \mathcal{E} \to \mathcal{E}$ such that 
$A_{x} = A (x,.) : \mathcal{E}_{x} \to \mathcal{E}_{f(x)}$
is a conformally symplectic linear isomorphism in the following sense.
There are cordinates such that $A^{*}\omega=\mu \omega$
where $\omega$ is a 
the canonical symplectic $2$-form $\omega$ and  
$\mu : \Sigma \to \mathbb{R}$
is a positive function.
The application $A$ also satisfies
$\|A_{x}\|<\infty$ for all $x \in \Sigma$.
We call $(\Sigma, f, \mathcal{E}, A, \omega, \mu)$ a 
\emph{conformally symplectic linear system} or a 
\emph{conformally symplectic linear cocycle} over $f$.
When there is no ambiguity, we denote the linear system $(\Sigma, f, \mathcal{E}, A, \mu)$ 
only by $A$.
\end{dfn}

The norm of the application $A$, is defined as $\|A\| = \text{maximum}\{ sup_{x\in\Sigma} \|A_{x}\|, sup_{x\in\Sigma} \|A_{x}^{-1} \| \}$.

\begin{lm}
Let $(\Sigma, f, \mathcal{E}, A, \mu)$ and $(\Sigma, f, \mathcal{E}, B, \mu)$ be  conformally symplectic linear systems. Then $(\Sigma, f, \mathcal{E}, A \circ B, \mu)$ satisfies $\| A \circ B\| \leq \| A \| \| B \|$. 
\end{lm}

We denote $CS(\Sigma, f, \mathcal{E})$ the set of conformally symplectic linear systems
$(\Sigma, f, \mathcal{E}, A, \mu)$ for all $\mu$ and equip $CS(\Sigma, f, \mathcal{E})$ 
with the distance 
$d(A,B)= \text{maximum}\{ \|A-B\|, \|A^{-1} - B^{-1}\| \}$ 
for $A, B \in \mathcal{CS}(\Sigma, f, \mathcal{E})$.

\begin{lm}
Given $K>0$, $\varepsilon_{1}>0$ and a conformally symplectic linear system
$(\Sigma, f, \mathcal{E}, A, \mu)$ such that $\|A\|< K$ then for all 
symplectic $\varepsilon_{1}$-perturbation of the identity $(\Sigma, Id_{\Sigma}, \mathcal{E}, E, 1)$
the composition $E \circ A$ and $A \circ E$ are $\varepsilon$-perturbation of $A$ 
with $\varepsilon = \varepsilon_{1} K$. Furthermore, for all $\varepsilon$-perturbation 
$\tilde{A}$ of $A$ there exists a symplectic $\varepsilon_{1}$-perturbation of the identity
$(\Sigma, Id_{\Sigma}, \mathcal{E}, E, 1)$ such that $\tilde{A} = E \circ A$.
\end{lm}

\begin{proof}
The application $E \circ A$ satisfies 
$\| E \circ A - A\| = \| (E - Id_{\Sigma})\circ A\| \leq \| E - Id_{\Sigma} \| \| A\| < \varepsilon_{1}$ e $(E\circ A)^{*}\omega = A^{*} E^{*} \omega = \mu \omega$. 
Analogously, $A \circ E$ satisfies
$\| A \circ E - A\| = \| A \circ (E - Id_{\Sigma})\| \leq  \| A\| \| E - Id_{\Sigma} \|< \varepsilon$ and $(A \circ E)^{*}\omega=\mu \omega$.

Define $E=\tilde{A}\circ A^{-1}$, the $E \circ  A= \tilde{A}$ e $\|E-Id_{\Sigma}\| < \| \tilde{A} -  A \| \|A^{-1}\| < \varepsilon_{1}$
and $E^{*}\omega = (\tilde{A}\circ A^{-1})^{*}\omega
=(A^{-1})^{*} \circ \tilde{A}^{*}\omega = (A^{-1})^{*} (\mu \omega) = \mu \omega$.
\end{proof}

\begin{dfn}
The conformally symplectic linear system $(\Sigma, f, \mathcal{E}, A, \mu)$ is called a
\emph{conformally symplectic periodic linear system} if all $x \in \Sigma$
is a periodic point of $f$. We denote $p(x)$ the period of $x$.
\end{dfn}

\begin{dfn}
The conformally symplectic linear system $(\Sigma, f, \mathcal{E}, A, \mu)$ 
is called a \emph{conformally symplectic continuous linear system} 
if the fiber structure varies continuously and $A: \mathcal{E} \to \mathcal{E}$ is continuous.
\end{dfn}

\begin{dfn}
The conformally symplectic linear system $(\Sigma, f, \mathcal{E}, A, \mu)$ 
is called a \emph{conformally symplectic matrix system} 
if $\mathcal{E}=\Sigma \times \mathbb{R}^{n}$,
$\mathbb{R}^{n}$ is equipped with the euclidean canonical metric and there exists
a positive function 
 $\mu: \Sigma \to \mathbb{R}$ such that $A^{t} J A = \mu J.$  We denote  
the \emph{conformally symplectic matrix system} by $(\Sigma, f, A, \mu)$  
\end{dfn}

Given a set $\mathcal{A}$, a word with letters in $\mathcal{A}$
is a finite sequence of elements of $\mathcal{A}$. 
The size of this word is its number of letters. 
The set of words admits a natural semi group structure: 
the product of $[a]=(a_{1}, \dots, a_{n})$
by $[b]=(b_{1}, \dots, b_{n})$ is $[a][b]=(a_{1}, \dots, a_{n},b_{1}, \dots, b_{n})$.
We say a word $[a]$ is not a power if $[a] \neq [b]^{k}$
for all word $[b]$ with $k>1$. 

If $(\Sigma, f, \mathcal{A}, \mu)$ is a conformally symplectic matrix system
then for $x \in \Sigma$ we denote the word 
$$(A(f^{p(x)-1}(x)),\dots,A(x))$$ 
by $[M]_{A}(x)$.
The matrix $M_{A}(x)=A(f^{p(x)-1}) \circ \dots \circ A(x)$ is the product of letters of $[M]_{A}(x)$.

\begin{dfn}
Given $\varepsilon > 0$, a conformally symplectic periodic linear system 
$(\Sigma, f, \mathcal{E}, A, \mu)$ admits $\varepsilon$-\emph{transitions} if:
\begin{itemize}
\item[(i)] for all finite family of points
 $x_{1}, \dots, x_{n}=x_{1} \in \Sigma$ there exists a coordinate system over $\mathcal{E}$ 
such that we can consider this system as a conformally symplectic matrix system
 $(\Sigma, f, A, \omega, \mu)$ and
\item[(ii)] For all $(i,j) \in \{1, \dots, n\}^{2}$ there exists a finite word
$[t^{i,j}]$ of matrices in $CS(n,\mathbb{R})$ such that the word
$$
[W(\iota,a)] = [t^{i_{1},i_{m}}][M_{A}(x_{i_{m}})]^{a_{m}}[t^{i_{m},i_{m-1}}][M_{A}(x_{i_{m-1}})]^{a_{m-1}} \dots [t^{i_{2},i_{1}}][M_{A}(x_{i_{1}})]^{a_{1}}
$$
with $\iota = (i_{1},.., i_{m}) \in \{ 1,.., n \}^{m}$, 
$a = (a_{1},.., a_{m}) \in \mathbb{N}^{n}$ 
and the word 
$((x_{i_{1}}, a_{i_{1}}),.., (x_{i_{m}}, \alpha_{i_{m}}))$ 
with letters in $\Sigma \times \mathbb{N}$
is not a power.

Then there exists $x= x (\iota, a) \in \Sigma$ such that
\begin{itemize}
\item[(ii.1)] the size of $[W(\iota,a)]$ is $per(x),$
\item[(ii.2)] The word $[M]_{A}(x)$ is $\varepsilon$-close to $[W(\iota,a)]$ 
and there exists a $\varepsilon$-perturbation $\tilde{A}$ of
$A$ such that $[M]_{\tilde{A}(x)}=[W(\iota, a)],$
\item[(ii.3)] We can choose $x$ such that the distance from the orbit of $x$ 
and $x_{i}$, $i \in \{1, \dots, n-1\}$, is bounded by a function $\alpha_{i}$ 
which tends to zero when $a_{i}$ tends to infinity.
\end{itemize}
\end{itemize}
Given $\iota$ and $a$ as above, we denote $[t^{i,j}]$ an $\varepsilon$-\emph{transiton} from 
$x_{1}$ to $x_{j}$ and
we call an $\varepsilon$-\emph{transtition matrix} to the product $T_{i,j}$ of letters composing 
$[t^{i,j}]$.
\end{dfn}

\begin{dfn}
A conformally symplectic periodic linear system admits \emph{transitions} 
if for any $\varepsilon>0$ the system admits $\varepsilon$-transitions.
\end{dfn}

Below, the dictionary between a conformally symplectic linear system and the tangent bunddle dynamics of the Gaussian thermostat over an homoclinic class:

\begin{lm}
Given  a hyperbolic periodic saddle point $p$ of index $k$ (the dimension of the stable subbundle). 
The differential $df$ of the Poincar\'e application $f$
induces a continuous linear system with transitions on the set $\Sigma$ 
of hyperbolic saddles in $H(p, f)$ of index $k$ which are  homoclinically related to $p$.
\end{lm}

\begin{proof}
Fix $\varepsilon >0$ and a finite family $x_{1}, \dots, x_{n} \in \Sigma$.
The points $x_{i}$ are homoclinically related to $p$ then there exists a
transitive compact hyperbolic subset $K \subset H(p,f)$
which contains all points $x_{i}$. 
It is possible to cover $K$ by a Markov partition composed by rectangles $R_{k}$.
For each $x \in K$ we define conformally symplectic coordinates $\phi_{x}: U_{x} \to M$.
It is possible to refine the Markov partition such that $R_{k}$ is contained in a open set $U_{x}$ 
and for $x$ and $y$ at the same rectangle
$\|A(x)-A(y)\|<\varepsilon$ and $\|A(x)^{-1}-A(y)^{-1}\|<\varepsilon$.
For each rectangle $R_{k}$ we write $df$ on the conformally symplectic coordinates
and consider the associated matrix system $(K, f, A, \mu)$.
The transition from $x_{i}$ to $x_{j}$ is obtained by the property of 
Markov partition and the conformally symplectic perturbative theorem.
\end{proof}

On what follows it is showed that a property valid on a point of a conformally symplectic linear systems which admits transitions
is extended over a dense set. Next lemma is only true to the case
of strict conformally symplectic case, i.e., $\mu$ is not identically $1$ 
since  for $\mu \equiv 1$ it is impossible to create contractions or dilations:

\begin{lm}[Spreading property]
Let $(\Sigma, f, \mathcal{E}, A, \mu)$ be a conformally symplectic periodic linear system with transitions.
Fix $\varepsilon > \varepsilon_{0} > 0$, assume there exists  a $\varepsilon_{0}$-perturbation  $B$ of $A$
and $x \in \Sigma$ such that $[M]_{B}$ is a dilation (or a contraction).
Then
there exists an $f$-invariant dense set  $\Sigma_{C} \subset \Sigma$ 
and a $\varepsilon$-perturbation
$C$ of $A$ such that for all $y \in \tilde{\Sigma}$, $[M]_{C}(y)$ is a dilation 
(or a contraction).
\end{lm}

\begin{proof}
Let $\Sigma_{B}\subset\Sigma$ a dense subset and $\varepsilon_{1}=\varepsilon-\varepsilon_{0}$.

Fix $\delta>0$ and consider $z \in \Sigma$, the family of points
$x_{1}=z \in \Sigma_{B}$, $x_{2}=x$ e $x_{3}=z$ and the word
$$
[W] = [t^{1,2}][M]_{A}^{n}(x)[t^{2,1}][M]_{A}^{n(z,\delta)}(z) 
$$
the linear system $(\Sigma, f, \mathcal{E}, A, \mu)$ admit transitions so there exists 
$z_{n} \in \Sigma$ such that $d(z_{n},z)<\delta$ and a $\varepsilon_{1}$-perturbation
$\tilde{A}$ of $A$ such that $[M]_{\tilde{A}}(z_{n})=[W]$.

Consider a $\varepsilon_{0}$-perturbation $C$ of $\tilde{A}$ defined along the orbit of
$z_{n}$ by

$$
[M]_{C}(z_{n})=[t^{1,2}][M]_{B}^{n}(x)[t^{2,1}][M]_{A}^{n(z,\delta)}(z).
$$

With $n$ big enough, we have $[M]_{C}(z_{n})$ a dilation or a contraction and 
$\Sigma_{C}$ is defined as the union of $z_{n}$.
\end{proof}

% A vers�o simpl�tica da Propriedade de Espalhamento:

% \begin{lm}[Propriedade de Espalhamento]
% Seja $(\Sigma, f, \mathcal{E}, A)$ um sistema linear simpl�tico peri�dico com transi��es.
% Fixe $\varepsilon > \varepsilon_{0} > 0$ e assuma que existe $\varepsilon_{0}$-perturba��o $B$ de $A$
% e $x \in \Sigma$ tal que $[M]_{B}$ tem um autovalor complexo (que corresponde a uma rota��o irracional) ent�o
% existe um conjunto denso $f$-invariante $\tilde{\Sigma} \subset \Sigma$ e uma $\varepsilon$-perturba��o
% $C$ de $A$ tal que para todo $y \in \Sigma_{C}$, $[M]_{C}(y)$ tem autovalor complexo.
% \end{lm}

% \begin{proof}
% A demonstra��o � a mesma do lema anterior, exceto pelo seguinte fato:

% $$
% [M]_{C}(z_{n})=[t^{1,2}][M]_{B}^{n}(x)[t^{2,1}][M]_{A}^{n(z,\delta)}(z).
% $$

% Fazendo $n$ grande o suficiente, temos que $[M]_{C}(z_{n})$ tem um autovalor complexo. 

% \end{proof}

\begin{dfn}
Given a conformally symplectic linear system $(\Sigma, f, \mathcal{E}, A, \mu)$. 
A $\varepsilon$-perturbation $\tilde{A}$ of $A$
is a conformally symplectic linear system  $(\Sigma, f, \mathcal{E}, \tilde{A}, \mu)$
such that $d(\tilde{A},A)<\varepsilon$ and which preserves the same conformally symplectic  structure $\mu$, 
i.e., $\tilde{A}^{*}\omega=\mu \omega$. 
\end{dfn}

\begin{lm}\label{perturbacao}
Fix $\varepsilon>0$. Any $\varepsilon$-pertubation of a matrix
$A \in CS(n,\mathbb{R})$ can be written as a composition of $A$ with a symplectic matrix
$\varepsilon_{1}$-close to the identity $E$ and $\varepsilon=\varepsilon_{1}\| A \|$. 
The compositions of $A$ with $E$ a symplectic matrix $\varepsilon_{1}$-close to the identity,
 $A \circ E$ and $E \circ A$, are $\varepsilon$-perturbations of $A$. 
\end{lm}

\begin{proof}
Let $\tilde{A}$ be an  $\varepsilon$-perturbation of $A$ and the symplectic matrix 
$E = \tilde{A} \circ A^{-1}$ then $E \circ A = \tilde{A}$ and $\|E \circ A - A\|< \varepsilon$.

Let $E \in S(n,\mathbb{R})$ be a symplectic matrix 
$\varepsilon_{1}$-close to the identity then
$E \circ A$ and $A \circ E$ are $\varepsilon$-perturbations of $A$.
\end{proof}

\begin{dfn}
A conformally symplectic periodic linear system $(\Sigma, f, \mathcal{E}, B, \mu)$ 
of dimension $2n$ is diagonalizable if for all 
$x \in \Sigma$ the matrix $M_{B}(x)$ has real positive eigenvalues with multiplicity $1$.
We denote
$\lambda_{1}(x) \leq \dots \leq \lambda_{n}(x) \leq \lambda_{\overline{n}}(x) \leq \dots \leq 
\lambda_{\overline{1}}(x)$ the eigenvalues of $M_{B}(x)$ such that
$\lambda_{i}(x)\lambda_{\overline{i}}(x)=\mu(x)$ for $i=1,\dots , n$.
We also denote $E_{i}(x)$ the eingenspace corresponding to 
$\lambda_{i}(x)$ and $E_{\overline{i}}(x)$ the eingenspace corresponding to
$\lambda_{\overline{i}}(x)$. Then 
$\mathcal{E}=(\oplus_{1}^{n}E_{i}) \oplus (\oplus_{1}^{n}E_{\overline{i}})$
is a invariant decomposition of $B$.
\end{dfn}

\begin{lm} \label{diagonalizavel}
Let $(\Sigma, f, \mathcal{E}, A, \mu)$ be a conformally symplectic periodic linear system 
with transitions. Then for
$\varepsilon>0$ there exists diagonalizable $\varepsilon$-perturbation $B$ of $A$ 
defined on a dense set $\tilde{\Sigma} \subset \Sigma$.
\end{lm}

\begin{proof}
Given a matrix $A \in CS(n,\mathbb{R})$,
there exists an arbitrarily small perturbation $B \in CS(n, \mathbb{R})$ of $A$ such that $B$ 
has all its eigenvalues are positive real numbers and with multiplicity $1$.

This follow from \cite{jordansimpletica} 
which constructs a symplectic change of coordinates $T$ such that

$$
T^{-1}AT =\left[
\begin{array}{ccc|ccc}
A_{11} &   0    & 0       & A_{12} &    0   & 0 \\      
0      & \ddots & 0       & 0      & \ddots & 0 \\
0      &    0   & A_{n1}  & 0      &    0   & A_{n2}\\ 
\hline  
A_{13} &    0   & 0       & A_{14} &    0   & 0 \\ 
0      & \ddots & 0       & 0      & \ddots & 0 \\
0      &    0   & A_{n3}  & 0      &    0   & A_{n4}\\
\end{array}\right]
$$
where the submatrices 
$
\begin{bmatrix}
A_{i1} & A_{i2}\\
A_{i3} & A_{14}
\end{bmatrix}
$
are canonical blocks analogous to the Jordan blocks.
\end{proof}

\subsection{Restrictions and decompositions}
\label{restr}

An invariant subbundle $F$ is a collection of subspaces 
$F_{x} \subset \mathcal{E}_{x}$ such that $dim(F(x))=c$
$\forall x \in \Sigma$ and $A(F(x))=F(f(x))$.
A $A-$splitting $E=F \oplus G$ is given by invariant subbundles such that
$\mathcal{E}_{x} = E(x) \oplus G(x)$, $\forall x \in \Sigma$.

\begin{dfn}
Let $(\Sigma, f, \mathcal{E}, A, \mu)$ be a conformally symplectic linear system  and
a A-splitting $\mathcal{E}=F \oplus G$. We call this splitting a \emph{dominated splitting},
and we denote $F \prec G$, if there exists $l \in \mathbb{N}$  such that all 
$x \in \Sigma$:
$$
\|A^{l}(x)\big{|}_{F}\|\|A^{-l}(f^{l}(x))\big{|}_{G} \|< \frac{1}{2}.
$$
\end{dfn}

\begin{dfn}
Let $(\Sigma, f, \mathcal{E}, A, \mu)$ be a conformally symplectic linear system   and
a A-splitting $\mathcal{E} = \oplus E_{i}$ in
$k$ invariant subbundles $E_{1}, \dots, E_{k}$.
We say two subspaces $E_{m}$ e $E_{n}$,  $m,n \in \{ 1, \dots k \}$, have 
\emph{dominated splitting} associated to the subspaces $E_{m}$ e $E_{n}$ and we denote 
$E_{m} \prec E_{n}$ if there exists $l \in \mathbb{N}$  such that for all $x \in \Sigma$ 
the restriction of $A$ in each subspace satisfies:
$$
\|A^{l}(x)\big{|}_{E_{m}}\|\|A^{-l}(f^{l}(x))\big{|}_{E_{n}} \|< \frac{1}{2}.
$$
To emphasize $l$, we say $F \oplus G$ is a $l$-dominated splitting and 
we write $F \prec_{l} G$.
\end{dfn}

Differently from the general case, treaded in \cite{bdp}, we need to preserve the
conformally symplectic structure and thus
we can not make the restriction and quotient of a matrix in such a way
that the perturbation on the restriction has no influence on the quotient and vice versa.
We  replace the concept of restriction and quotient only by the concept of \emph{restriction}:
let $F_{1}, \dots, F_{k}$ be $k$ invariant subspaces  such that $\mathcal{E} = \oplus F_{i}$,
we can write $A$ as:

\begin{equation*}
\mathbf{A} = \left(
\begin{array}{cccc}
A_{1}  & 0      &  \ldots & 0 \\
0      & A_{2}  &  \ldots & 0 \\
\vdots & \vdots &  \ddots & \vdots \\
0      & 0      &  0      & A_{k}
\end{array} \right)
\end{equation*}
such that  $A_{i}(F_{i}) \subset F_{i}$.

The next lemma corresponds to lemma 4.4 in \cite{bdp} and it uses strongly the symplectic conformal structure. 
\begin{lm} \label{propdecomposicao}
Given $K>0$,  $l \in \mathbb{N}$ and a linear system
$(\Sigma, f, \mathcal{E}, A, \mu)$ bounded by $K$ 
with invariant decomposition $E \oplus F \oplus G$
then there exists $L$ such that it holds:
\begin{itemize}
\item[(i)] if $E \prec_{l} F$ and $E \prec_{l} G$ then $E \prec_{L} (F \oplus G)$.
\item[(ii)] if $F \prec_{l} G$ and $E \prec_{l} G$ then $(E \oplus F) \prec_{L}  G$.
\end{itemize}
\end{lm}

\begin{proof}
Using the conformally symplectic basis we can write the matrices $A$ as:
\begin{equation*}
A = \left(
\begin{array}{ccc}
A_{E}  & 0      &  0 \\
0      & A_{F}  &  0 \\
0      & 0      &  A_{G}
\end{array} \right)
\end{equation*}
Observe that the entry of the second line and third column is zero, differently to the proof of lemma 4.4 in \cite{bdp} where that entry corresponds to a non-zero matrix $B(x)$. To conclude, the proof follows  as in lemma 4.4 in \cite{bdp} but now using  that $B(x)=0$.
\end{proof}

\begin{dfn}
Consider a $f$-invariant set $\Sigma^{\prime} \subset \Sigma$
and the restriction of $\mathcal{E}$ over $\Sigma^{\prime}$. 
The conformally symplectic linear system induced by $A$, 
$(\Sigma^{\prime}, f|_{\Sigma^{\prime}}, \mathcal{E}_{\Sigma^{\prime}})$, 
is called a \emph{subsystem induced by $A$ in $\Sigma^{\prime}$}.
\end{dfn}
The next lemma correspond to lemma 1.4  in \cite{bdp} adapted to the context of conformally symplectic  linear systems. The statement in \cite{bdp} holds for any linear systems so the proof of next lemmas follows as a corollary.
\begin{lm}\label{densodd}
Let $(\Sigma, f, \mathcal{E}, A, \mu)$ be a conformally symplectic continuous linear system
such that there exists an invariant dense set
$\Sigma_{1} \subset \Sigma$ which subsystem admits $l$-dominated splitting,
then $(\Sigma, f, \mathcal{E}, A, \mu)$
admits an $l$-dominated splitting.
More generally, suppose  there exists a sequence of subsystems
$(\Sigma, f, \mathcal{E}, A_{k}, \mu)$ converging to  $(\Sigma, f, \mathcal{E}, A, \mu)$
such that for all $k$ there exists a dense invariant subset 
$\Sigma_{k} \subset \Sigma$ where $A_{k}$ admits $l$-dominated splitting. 
Then $A$ admits $l$-dominated splitting on all $\Sigma$.
Finally, a dominated splitting over a conformally symplectic continuous linear systems
is continuous.
\end{lm}

\begin{lm}\label{subespacos}
Let $(\Sigma, f, \mathcal{E}, A, \mu)$ be a conformally symplectic periodic linear system 
with $\varepsilon$-transi\-tions and with dominated splitting
$E_{1} \prec \dots \prec E_{n}$. Fix $\varepsilon_{0}>\varepsilon$.
Then given two points $x_{i}$ and $x_{j} \in \Sigma$, $k \in \{1, \dots, m \}$ and 
$[t^{ji}]$ $\varepsilon$-transition between $p_{i}$
and $p_{j}$. Then there exists $\varepsilon_{0}$-transition $[\tilde{t}^{ji}]$ 
such that $\tilde{T}_{ji}$ takes $E_{k}(x_{i})$ in $E_{k}(x_{j})$ .
\end{lm}

\begin{proof}
The application $T_{ji}$ satisfies 
$angle(\tilde{T}_{ij}(E_{k}(x_{i})), E_{k}(x_{j})) < \varepsilon$.
Consider a perturbation of $\tilde{T}_{ji}$ such that 
$\tilde{T}_{ji}(E_{k}(x_{i}))$ does not have components in $E_{m}(x_{i})$, $m>k$.
Due to the dominated splitting, the image of $E_{k}(x_{i})$ by
$M_{A}^{n}(x_{j})\tilde{T}_{ji}$ is arbitrarily close to 
$E_{k}(x_{j})$ when $n$ is big enough.
By a perturbation of $M_{A}$ with a small rotation we have that $M_{\tilde{A}}^{n}(x_{j})\tilde{T}_{ji}(E_{k}(x_{i}))=E_{k}(x_{j})$.
\end{proof}

The next lemma needs the definition of the rank of an eigenvalue.

\begin{dfn}
Let $M \in GL(n,\mathbb{R})$ such that $M$ has a complex eigenvalue $\lambda$. 
We say that $\lambda$ has rank $(i,i+1)$ if there exists an $M$-invariant splitting 
of $\mathbb{R}^{n}$, $F \oplus G \oplus H$, such that:
\begin{itemize}
\item[(i)] the eigenvalues $\sigma$ of $F$ satisfies $|\sigma|<|\lambda|,$
\item[(ii)] the eigenvalues $\sigma$ of $H$ satisfies $|\sigma|>|\lambda|,$
\item[(iii)] $dim(F)=i-1$ and $dim(F)+dim(G)+dim(H)=n$,
\item[(iv)] the plane $G$ is the eigenspace of $\lambda$.
\end{itemize}
\end{dfn}

\begin{lm}\label{invcomplexo}
Fix $\varepsilon>0.$
Let $(\Sigma, f, \mathcal{E}, A, \mu)$ be a conformally symplectic linear system and
$p \in \Sigma$ a periodic point with complex eigenvalue of rank
$(i,i+1)$ associated to the subspace $E_{i} \oplus E_{i+1}$.
Then there exists $\varepsilon$-perturbation $\tilde{A}$ of $A$ and $m \in \mathbb{N}$ 
such that $[M]_{A}^{m}(p)(E_{i+1}) = E_{i}$.
\end{lm}

\begin{proof}
We perform a $\frac{\varepsilon}{2}$-perturbation $\tilde{A}$ of $A$ 
such that the complex eigenvalue of rank $(i,i+1)$ of $\tilde{A}$
has irrational entries. 
Then there exists $m$ such that the angle between $[M]_{\tilde{A}}^{m}(p)(E_{i+1})$
and $E_{i}$ is equal to $\alpha < \frac{\varepsilon}{2}$. 
Composing  $\tilde{A}$ with a rotation by $\alpha$ (which is symplectic), we conclude  the result.
\end{proof}

\begin{lm}\label{maiorautovalor}
Given a vector $v$ and two conformally symplectic matrices 
$T$ and $M$ such that the vector $w$ is the eigenvector associated to the largest 
eigenvalue in modulus of $M$, 
then there exist $n$ and a  $\varepsilon$-perturbation $\tilde{B}$ of $B = M^{n} \circ T$ 
such that  $\tilde{B}(v)=w$.
\end{lm}

\begin{proof}
If $v$ is an invariant vector, we perform a symplectic perturbation
to make $v$ not invariant.
Hereafter, we iterate $M$ to make the angle between $Bv = M^{n} \circ T(v)$ and $w$ 
equal to $\alpha < \varepsilon$.
We compose $B$ with a rotation by $\alpha$ and then we have the result.
\end{proof}

\subsection{Bidimensional case} \label{bidi}

The proof of the next two results follows from an argument of R. Ma\~n\'e in \cite{ma2} 
which is also valid in the conformally symplectic case.

\begin{prop}\label{bidimensional}
Given $K>0$ and $\varepsilon>0$ there exists $l\in \mathbb{N}$ such that for any two dimensional  conformally symplectic 
linear system $(\Sigma, f, \mathcal{E}, A, \mu)$ with 
norm bounded by $K$ and  matrices $M_{A}(x)$ preserving orientation, one 
 one of the following possibilities holds:
\begin{itemize}
\item[(i)] $A$ admits $l$-dominated splitting.
\item[(ii)] There exists $\varepsilon$-perturbation $\tilde{A}$ of $A$ and $x \in \Sigma$ 
such that $M_{\tilde{A}}(x)$ has a complex eigenvalue.
\end{itemize}
\end{prop}

We denote with $R^{\theta} : \mathbb{R}^{2} \to \mathbb{R}^{2}$ the rotation by the angle $\theta$.

\begin{lm}\label{angulo}
For all $\alpha > 0$ and all matrix $M \in GL_{+}(2,\mathbb{R})$ with two eigenspaces
$E_{1}$ and $E_{2}$ whose angle is less than $\alpha$ then there exists
$s \in [-1,1]$ such that $R^{s\alpha} \circ M$ has a complex eigenvalue.
\end{lm}

\subsection{Generic Dichotomy}\label{generic d}

To prove the main theorem, it suffices to proof the analogous 
result for conformally symplectic linear systems, 
once we reduced the problem to this case.

\begin{prop} \label{reducao}
For $K>0$, $n>0$ and $\varepsilon>0$ there exists $l>0$ 
such that a conformally symplectic periodic linear system $(\Sigma, f, \mathcal{E}, A, \mu)$ 
bounded by $K$ and with transitions  satisfying one of the next possibilities:
\begin{itemize}
\item[(i)] $A$ admits a $l$-dominated splitting.
\item[(ii)] There exists a $\varepsilon$-perturbation $\tilde{A}$ of $A$ and a point
$x \in \Sigma$ such that $M_{A}(x)$ is an homothety.
\end{itemize}
\end{prop}

The proof is divided into two steps: 
the first one says if we can not create complex eigenvalues of rank $(i,i+1)$ 
for $i \in \{1, \dots, 2n-1\}$
then we have dominated splitting and the second one says if we can 
create such eigenvalues then 
we can perturb the original linear system with transition to get homotheties.

\begin{prop}\label{semcomplexo}
Fix $\varepsilon >0$, $n \in \mathbb{N}$, and $K>0$.
Let $(\Sigma, f, \mathcal{E}, A, \mu)$ be a conformally symplectic continuous linear system
with transitions, of dimension $2n$ bounded by $K$ 
such that all $\varepsilon$-perturbation of $A$ does not admit complex eigenvalues
of rank $(i,i+1)$ for $i = 1, \dots, 2n-1 $.
Then there exists $l \in \mathbb{N}$ such that $(\Sigma, f, \mathcal{E}, A, \omega, \mu)$ 
admits $l$-dominated splitting
$\mathcal{E} = F \oplus G$, $F \prec_{l} G$, with $dim(F)=i$
\end{prop}

The proof of previous proposition is given in subsection \ref{sufi}.

\begin{prop}\label{comcomplexo}
Fix $\varepsilon > \varepsilon_{0} >0$. 
Let $(\Sigma, f, \mathcal{E}, A, \mu)$ be a conformally 
symplectic periodic linear system with transitions.
Suppose for all $i = 1, \dots, 2n-1 $ there exists $\varepsilon_{0}$-perturbation of $A$
with complex eigenvalues of rank $(i,i+1)$.
Then there exists $\varepsilon$-perturbation $\tilde{A}$ of $A$ and $x \in \Sigma$ 
such that $M_{\tilde{A}}(x)$ is a homothety.
\end{prop}

\subsection{Sufficient condition for  $l$-dominated splitting}\label{sufi}

In this section, we prove the proposition \ref{semcomplexo} using  the next three lemmas which 
assure us that 
if a conformally symplectic periodic diagonalizable linear system with transitions
has no complex eigenvalues of rank $(i,i+1)$ then there exists
dominated splitting $E \prec F$ such that $dim(E)=i$.
In  lemma \ref{adaptado} it is proved that  
provided  two invariant subspaces, $E_{j}$ and $E_{k+1}$, if they 
do not have a $l$-dominated splitting then it is possible to create a complex eigenvalue of rank
$(i, i+1)$ for $i \in \{ j, \dots, k\}$.
The idea of the proof is to restrict the application $A$ 
to the subspaces $E_{j}$ and $E_{k+1}$ and to use 
the bidimensional results observing that all the perturbation are performed in the
conformally symplectic space. This means that when we perturb the eigenvalues associated to
the eigenvectors $\lambda_{j}$ and $\lambda_{k+1}$ then  this perturbation affects
all the subspaces $E_{j}$, $E_{k+1}$, $E_{\overline{j}}$ e $E_{\overline{k+1}}$.
In lemma  \ref{somak} we prove that 
if we do not have complex eigenvalue of rank $(i,i+1)$
then there exists $L_{1}$ such that there is a  $L_{1}$-dominated restricted  to
the sum of subspaces $E_{k+1}$.
Finally, lemma \ref{somaj} says that there exists $L_{2}$ such that 
the $L_{2}$-dominated splitting is still valid for the sum of $E_{k}$.
With this three lemmas, the proof of proposition \ref{semcomplexo} 
follows from lemmas \ref{diagonalizavel} and \ref{densodd}.

For the proof we will need the rotation $R_{ij}^{\theta} : \mathbb{R}^{2n} \to \mathbb{R}^{2n}$:

$$
R^{\theta}_{i,j} =
\begin{array}{cc}
\begin{array}{c}
\\
\end{array} &
\begin{array}{cccccccccc}
i&&&&&&&&&j\\
\end{array} \\
\begin{array}{c}
\\
i\\
j\\
\\
\end{array} &
\left( \begin{array}{ccccccccc}
1&\dots&0                 &0&\dots&0&0                   &\dots&0\\
0&\dots&\cos (\theta) &0&\dots&0&  -\sin (\theta) &\dots&0\\
&&                 &&\vdots&&                   &&\\
0&\dots&\sin (\theta) &0&\dots&0&  \cos (\theta) &\dots&0\\
0&\dots&0                 &0&\dots&0&0                   &\dots&1\\
\end{array} \right)
\end{array}
$$

\begin{lm}\label{adaptado}
Let $(\Sigma, f, \mathcal{E}, B, \mu)$ be  a conformally symplectic diagonalizable linear
system of dimension $2n$ and bounded by $K>0$.
Given $\varepsilon>0$ and $1 \leq i \leq 2n-1$, there exists $l\in \mathbb{N}$ 
such that one of the alternatives is valid:
\begin{itemize}
\item[(i)] There exists $\varepsilon$-perturbation of $B$ with complex eigenvalue of rank $(i,i+1),$
\item[(ii)] $E_{j} \prec_{l} E_{k+1}$ for all $j \leq i \leq k \leq 2n-1.$
\end{itemize}
\end{lm}

\begin{proof}
Recall that
$\lambda_{i}\lambda_{n+i}=\mu$  for $i = 1,\dots,n$ 
and $\bar{\lambda}_{i} = \lambda_{ (i + n - 1)(mod 2n) + 1}.$
Let $l$ be the constant of domination given by proposition \ref{bidimensional}.
If $E_{j} \prec_{l} E_{k+1}$ for all $j \leq i \leq k \leq 2n-1$ then it is done. 
Moreover, by the proposition \ref{angulo},
the angle between $E_{j}$ and $E_{k+1}$ is less than $\alpha$ 
and we can perturb $B$ to obtain a complex eigenvalues associated to the subspaces
$E_{j}$ and $E_{k}$ using the rotation $R_{jk}^{\alpha}$.
In order to preserve the conformally symplectic structure, we define the isotopy
$$
\tilde{B}_{t} = R_{\overline{j},\overline{k+1}}^{t\alpha} \circ R_{j,k+1}^{t\alpha} \circ B.
$$
By continuity, there exists $t_{0} \in [0,1]$ such that one of this alternatives is valid:
\begin{itemize}
\item $\lambda_{j}^{t}=\lambda_{i+1} \leq \lambda_{k+1}^{t}$,
\item $\lambda_{j}^{t} \leq \lambda_{i} = \lambda_{k+1}^{t}$,
\item $\lambda_{i} <  \lambda_{j}^{t} = \lambda_{k+1}^{t} < \lambda_{i+1}$,
\item $\lambda_{\overline{k+1}}^{t}=\lambda_{i+1} \leq \lambda_{\overline{j}}^{t}$,
\item $\lambda_{\overline{k+1}}^{t} \leq \lambda_{i} = \lambda_{\overline{j}}^{t}$,
\item $\lambda_{i} < \lambda_{\overline{k+1}}^{t} = \lambda_{\overline{j}}^{t} < \lambda_{i+1}.$
\end{itemize}
The last three possibilities are related to the cases $j \leq n$,  $k+1>n$ e $j+n \neq k+1$.
In all cases a small perturbation produces a complex eigenvalue of rank $(i,i+1)$.
Considering the difference between the general case and the conformally symplectic case
we can create besides the eigenvalue of rank $(i,i+1)$
associated to the subspaces $E_{j}$ e $E_{k+1}$, a
complex eigenvalue associated to the subspaces $E_{\overline{j}}$ e $E_{\overline{k+1}}$

\begin{center}
\begin{tikzpicture}
\draw (1,0) -- (9.7,0);
\draw[dashed] (9.8,0) -- (10.3,0);
\draw (10.3,0) -- (11,0);
\foreach \x in {1.5,2.5,3.5,4.5,5.5,6.5,7.5,8.5,9.5}
{
\draw (\x,0) -- (\x,0.1);
\draw (\x,0.1); 
}
\node[anchor=south] at (1.5, 0.1) {$\lambda_{1}$};
\node[anchor=south] at (2.5, 0.1) {$\lambda_{2}$};
\node[anchor=south] at (3.5, 0.1) {$\lambda_{3}$};
\node[anchor=south] at (4.5, 0.1) {$\lambda_{4}$};
\node[anchor=south] at (5.5, 0.1) {$\lambda_{5}$};
\node[anchor=south] at (6.5, 0.1) {$\lambda_{6}$};
\node[anchor=south] at (7.5, 0.1) {$\lambda_{7}$};
\node[anchor=south] at (8.5, 0.1) {$\lambda_{8}$};
\node[anchor=south] at (9.5, 0.1) {$\lambda_{9}$};
\draw (10.5,0) -- (10.5,0.1);
\draw (10.5,0.1) node[anchor=south]  {$\lambda_{2n}$};

\draw (4,0.75) rectangle (5,1.25);
\draw (5,0.75) rectangle (6,1.25);
\draw (4.5,1) node {$i$};
\draw (5.5,1) node {$i+1$};
\draw (1.1,-0.1) -- (1.1,-0.3) -- (4.9,-0.3) -- (4.9,-0.1);
\draw (5.1,-0.1) -- (5.1,-0.3) -- (10.9,-0.3) -- (10.9,-0.1);
\draw (3.5,-0.5) node {$A$};
\draw (7.5,-0.5) node {$B$};
\end{tikzpicture}
\end{center}

\end{proof}

\begin{lm}\label{somak}
Let $(\Sigma, f, \mathcal{E}, B, \mu)$ be 
a conformally symplectic diagonalizable linear system of dimension $2n$ and bounded by $K>0$.
Given $\varepsilon>0$ there exists $l_{0} \in \mathbb{N}$ such that for $1 \leq i \leq 2n-1$ 
is valid one of the alternatives below:
\begin{itemize}
\item[(i)] $B$ admits $\varepsilon$-perturbation with complex eigenvalue of rank $(i,i+1),$
\item[(ii)] for all $j \leq i$, $E_{j} \prec \oplus_{i+1}^{2n}E_{k}$.
\end{itemize}
\end{lm}

\begin{proof}
The proof is an inductively application of the lemmas \ref{propdecomposicao} and \ref{adaptado}.
\end{proof}

\begin{lm}\label{somaj}
Let $(\Sigma, f, \mathcal{E}, B, \mu)$ be
a conformally symplectic diagonalizable linear system of dimension $2n$ and bounded by $K>0$.
Given $\varepsilon>0$ there exists $L \in \mathbb{N}$ such that for $1 \leq i \leq 2n-1$ 
is valid one of the alternatives below:
\begin{itemize}
\item[(i)] $B$ admits $\varepsilon$-perturbation with complex eigenvalue of rank $(i,i+1)$,
\item[(ii)] $\oplus_{1}^{i}E_{k} \prec_{L} \oplus_{i+1}^{2n}E_{k}$.
\end{itemize}
\end{lm}

\begin{proof}
The proof is an inductively application of the lemmas \ref{propdecomposicao} and \ref{somak}.
\end{proof}

\section{Proof of Theorem C (Franks' Lemma)}\label{teofranks}

In  \cite{franks}, Franks proved that given 
a $C^{1}$ diffeomorphism $f: M \to M$  
over a Riemannian manifold $(M, g)$ and
 $\varepsilon > 0$, if we take a periodic point $x \in M$,
we can perform a $C^{1}$ small perturbation $g$ of $f$ such that
$g^{n}(x)=f^{n}(x)$, $n \in \mathbb{Z}$, and $dg^{n}$ is any isomorphism 
$\varepsilon$-close of $df^{n}$, for $n \in \mathbb{Z}$.
This result is known as \emph{Franks lemma}. The key point in that lemma is since is only required that $g$ is $C^1-$close to $f$, the support of the perturbation can be done arbitrary small in such a way that the perturbation preserves the trajectory.

For geodesic flows, an analogous result requires a different technique  
 from that used by Franks, since
geodesic flow perturbations are metric perturbations and these perturbations  are not local as was described above. In fact, to perturb a metric on a neighborhood of a closed geodesic
means that it is performed a perturbation in a cylinder in the tangent space 
where the geodesic flow is defined.
However, and overcoming such difficulty, in \cite{contreras}, Contreras proved a version of the Franks lemma
to the geodesic flow.

For a Gaussian thermostat, the situation is the same as the geodesic flow.
In this case, we  perturb the  vector field $E$ and similarly, 
a perturbation of $E$ in a neighborhood of the manifold
implies a perturbation over a cylinder in the tangent space
where the flow is defined.
In this section, we prove a version of the Franks lemma
for Gaussian thermostats adapting the ideas of Contreras in \cite{contreras}.

Let $(M, g)$ be a Riemannian manifold of dimension $n+1$.
We denote with $\mathscr{X}^{k}(M)$  the space of vector fields over $M$
of class $C^{k}$ and with
$\mathcal{R}^{l}(M)$  the space of Riemannian metrics over $M$
with $C^{l}$ topology. 

Let $(M, g, E)$ be
the Gaussian thermostat defined over $(M,g)$ 
with $E \in \mathscr{X}_{g}(M) \cap \mathscr{X}^{r}(M)$  and $\eta$ a closed orbit of $(M, g, E)$.
Our goal is to find a perturbation of $E$ which has $\eta$ as a closed orbit and, 
moreover, the transversal derivative cocycle associated to a point $p$ in $\eta$ (i.e., its  linear Poincar\'e application) 
is any conformally symplectic application close to the original one.
We consider on $\mathscr{X}^{r}(M)$ the topology defined by the open sets
$B(E, r) = \{ \tilde{E} \in \mathscr{X}^{r}(M) | max\{ \| \tilde{E} - E \|_{C^{1}} < r \} \}$.

The main theorem of this section is:

\begin{teoc}
Let $E \in \mathscr{X}_{g}(M)\cap \mathscr{X}^{r}(M)$, $4 \leq r \leq \infty$,  a periodic orbit $\eta$ of the Gaussian thermostat $(M,g,E)$, 
a point $\theta$ in $\eta$   and 
$T : \hat{T}_{\theta}SM \to \hat{T}_{\theta}SM$ 
its  transversal derivative cocycle.

Fixed $\varepsilon> 0$ there exists 
$\delta>0$ such that given a conformal linear map $L$ 
$\| L - T \| < \delta$ then
there exists $\tilde{E} \in \mathscr{X}_{g}(M)$ with $d( \tilde{E},  E)_{C^{1}} < \varepsilon$
which defines the Gaussian thermostat $(M,g,\tilde{E})$ and 
such that it has $\eta$ as an orbit and its  transversal derivative cocycle at $\theta$
is $L$.
\end{teoc}

The idea of the proof goes as follows.
Consider $(M, g, E)$ the Gaussian thermostat associated to the vector field
$E \in \mathscr{X}_{g}(M) \cap \mathscr{X}^{r}(M)$.
Let $\eta$ a periodic orbit of $(M, g, E)$, $p \in M$
a point of $\eta$
and $T: \hat{T}_{p}M \longrightarrow \hat{T}_{p}M$ the linear Poincar\'e application 
associated to $p$.
We consider the application  
$$S: \mathscr{X}_{g}(M) \cap \mathscr{X}^{r}(M)\longrightarrow CS(n)$$
whose domains are the vector fields of class $C^{r}$ which defines a closed $1$-form $\gamma$ 
and the image are the conformally symplectic matrices such that $S(E)=T$.

We consider subsets in $\mathscr{X}_{g}(M)$ satisfying the following:
\begin{itemize}
\item[$\mathcal{G}_{1}$]

For every orbit segment of size $\frac{r_{inj}}{2}$, where $r_{inj}$ is the injectity radius of $(M, g)$, has at least one point
such that the matrix $K_{g}+B_{g}$ has all distinct eigenvalues, where  $K$ is the curvature matrix and $[B]_{ij}= \frac{\partial E_{i}}{\partial x_{j}}$ 
is the derivative matrix of 
$E \in \mathscr{X}_{g}(M) \cap \mathscr{X}^{1}(M)$,

% every orbit segment has a point with the eigenvalues of $K+B$ are all distint.
\item[$\mathcal{G}_{2}$]
Provided a finite segment of  orbit of the Gaussian thermostat,  we consider the set of vector fields $C^1-$close to $E$ such that they coincides with $E$ on the self intersections of the given finite segment of orbit. 

\end{itemize}

Let $\mathcal{U}$ be a neighborhood of $E$, we prove that the image by $S$ of
$\mathcal{U} \cap \mathcal{G}_{1} \cap \mathcal{G}_{2}$
contains a ball of radius $\delta > 0$.
This means that any $\delta$-perturbation of the transversal derivative cocycle $T$ 
can be performed by a perturbation in
$\mathcal{U} \cap \mathcal{G}_{1} \cap \mathcal{G}_{2}$.

We do this by partitioning the orbit $\eta$ in some segments $\eta_{i}$, $i = 1, \dots, n$.
For each segment, we consider the transversal derivative linear cocycle on the 
extremal points and we define the applications 
$S_{i} : \mathscr{X}_{g}(M) \cap \mathscr{X}^{r}(M) \to CS(M)$ which satisfies
$S_{i} (E) = T_{i}$ where $T_{i} : \hat{T}_{\eta_{i}^{0}}M \to \hat{T}_{\eta_{i}^{1}}M$
is transversal derivative linear cocycle on the extremal points
$\eta_{i}^{0}$ and $\eta_{i}^{1}$ of $\eta_{i}$. 

Using the Jacobi equation for Gaussian thermostats, we show $\|dS_{E}\| > d > 0$
and this implies the result.

We consider $\pi : TM \longrightarrow M$ the canonical projection.
% and $r_{inj}$ the injectity radius of $(M, g)$. 
We denote $\mathcal{S}(n)$ 
the symmetric $n \times n$ matrices, 
$\mathcal{S}^{*}(n)$ the symmetric $n \times n$ matrices with null diagonal
and $\mathcal{O}(n)$ the orthogonal $n \times n$ matrices.

Before we prove the theorem we need some preliminary:

\subsection{Generic condition}

Given $g \in \mathcal{R}^{2}(M)$, we define
\begin{itemize}
\item
$K_{g} : SM \to \mathcal{S}(n)_{\diagup \mathcal{O}(n)}$ 

$$
K_{g} (\theta)_{ij}  = - \langle R_{g}(\theta, e_{i})\theta, e_{j} \rangle_{\pi(\theta)}.
$$
\item
$B_{g}: SM \to \mathcal{S}(n)_{\diagup \mathcal{O}(n)}$
$$
B_{g}(\theta)_{ij} =  \frac{\partial E_{i}}{\partial e_{j}^{g}} (\pi (\theta)),
$$
\end{itemize}
where
$\{ \theta, e_{1}=e_{1}(0),.., e_{n}=e_{n}(0) \}$
is an orthonormal basis of $T_{\phi^{E}_{0}(\theta)}M$
and $\{ \theta(t), e_{1}(t),.., e_{n}(t) \}$ is the parallel transport of this 
basis along 
$c(t)= \pi \circ \phi^{E}_{t}\Big{(}\frac{\theta}{|\theta|_{g}}\Big{)}$
according to the Riemannian connection associated to the metric $g$.
Let us remember the definition of $\mathcal{G}_{1}$:

\begin{eqnarray*}
\mathcal{G}_{1} &=& \{ E \in \mathscr{X}_{g}(M) \cap \mathscr{X}^{1}(M) | \\
&& \text{every orbit segment of size }\frac{r_{inj}}{2} \text{ has at least one point}\\
&& \text{such that the matrix } K_{g}+B_{g} \text{ has all distinct eigenvalues}\}.
\end{eqnarray*}

The following theorem states that $\mathcal{G}_{1}$ is generic.

\begin{teo} \label{autovalores}
The set $\mathcal{G}_{1}$ 
is open in $\mathscr{X}_{g}(M) \cap  \mathscr{X}^{1}(M)$ and 
$\mathcal{G}_{1} \cap \mathscr{X}^{\infty}(M)$ is dense in $\mathscr{X}_{g}(M)$.
\end{teo}

The proof only differs from the proof in \cite{contreras} by one aspect:
now we need the eigenvalues of $K_{g}+B_{g}$ to be  distinct instead of the eigenvalues of 
$K_{g}$. Since the proof is similar, we leave it to the reader.

\subsection{Perturbative theorem}

We begin this section with the definition of a prime orbit:

%\begin{dfn}
%A closed orbit $\eta$ is called \emph{prime} if it is not the iterate of
%any closed orbit of smaller period.
%\end{dfn}

Given a prime orbit $\eta$, let

\begin{itemize}
\item $W \subset M$ bea tubular neighborhood of $c = \pi \circ \eta,$
\item $\tau > 0$ such that $m \tau = period(\eta)$ with $m \in \mathbb{N}$ and
$\tau < r_{inj}$. %where $r_{inj}$ is the injectivity radius of $M$.
Consider, for $ 0 \leq k < m$, $\eta_{k}(t) = \eta( (t + k)\tau)$ with $t \in [0, 1]$.
We call $c_{k}$ the projection of $\eta_{k}$ in $M$, $c_{k}= \pi \circ \eta_{k}$.
\end{itemize}

For each segment $\eta_{k}$, we define the application
$S_{k} : \mathscr{X}_{g}(M) \cap \mathscr{X}^{r}(M) \longrightarrow CS(n)$ 
by
$$
S_{k}( E) = T^{E}_{k}
$$
where $T^{E}_{k}$ is the transversal linear cocycle between $\eta_{k}(0)$ and $\eta_{k}(1)$.

It can happen that the segments $c_{k}$ transversally intersect on $M$.
Given a segment $c_{0}$ the set of segments  which intersect it  is noted
$\mathcal{F}_{0} = \{ c_{1}, \dots, c_{l} \}$.

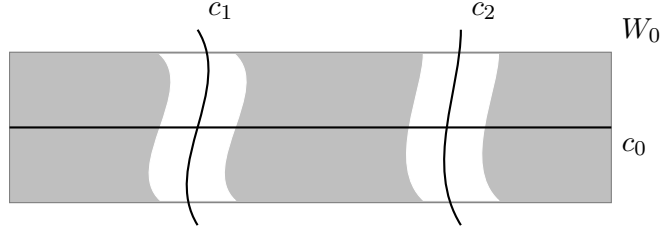
\begin{center}
\begin{tikzpicture}
\filldraw[gray!50] (0, -1) rectangle (8, 1) node[anchor=south west, black] {$W_{0}$};
\draw[gray] (0, -1) rectangle (8, 1); %node[anchor=south west, black] {$W_{0}$};
\filldraw[white] (2,-0.97)   .. controls (1.5,-0.5) and (2.5,0.5) .. (2,0.97) -- (3, 0.97)    .. controls (3.5,0.5) and (2.5,-0.5) .. (3,-0.97) -- cycle;
\filldraw[white] (5.5,-0.97) .. controls (5,-0.5) and (5.5,0.5) .. (5.5,0.97) -- (6.5, 0.97)  .. controls (6.5,0.5) and (6,-0.5) .. (6.5,-0.97) -- cycle;
\draw[thick] (0,0) -- (8,0) node[anchor=north west] {$c_{0}$};
%\draw[very thick] (0,-1.1) -- (0,1.1);
%\draw[very thick] (8,-1.1) -- (8, 1.1);
\draw[thick] (2.5,-1.3) .. controls (2,-0.5) and (3,0.5) ..(2.5,1.3) node[anchor=south west] {$c_{1}$};
\draw[thick] (6,-1.3) .. controls (5.5,-0.5) and (6,0.5) .. (6,1.3) node[anchor=south west] {$c_{2}$};
\end{tikzpicture}
\end{center}

\begin{dfn}
Let $c_{0}$ be an orbit segment and
$\mathcal{F}_{0}$ the set of orbit segments which intersect $c_{0}$.
Let $W_{0} \subset W$ be a tubular neighborhood of $c_{0}$ 
such that the segments of $\mathcal{F}_{0}$ are not contained in $W_{0}$.
We denote by $\mathcal{G}_{2}(\eta_{0}, W_{0}, \mathcal{F}_{0})$ 
the set of vector fields $\tilde{E} \in \mathscr{X}^{r}(M)$
such that
\begin{itemize}
\item $supp( \| \tilde{E} - E \| ) \subset W_{0},$
\item $\tilde{E} = E$ in a neighborhood $U_{0}$ of $\mathcal{F}_{0}.$
\end{itemize}
\end{dfn}

For the segment $\eta_{0}$, after we take $\mathcal{F}_{0}$ and $W_{0}$
we also have the set $\mathcal{G}_{2}(\eta_{0}, W_{0}, \mathcal{F}_{0})$ and
we apply the following result, which is the perturbation theorem on an orbit segment:

\begin{teo}
\label{segmento}
let $E \in \mathcal{G}_{1} \cap  \mathscr{X}_{g}(M) \cap \mathscr{X}^{r}(M)$, 
$4 \leq r \leq \infty$, 
and $\eta_{0}$ an orbit segment associated to the Gaussian thermostat $(M, g, E)$.
Given a neighborhood $\mathcal{U} \in \mathscr{X}_{g}(M) \cap \mathscr{X}^{1}(M)$,
there exists $\delta_{0} = \delta_{0}(g, E, \mathcal{U})>0$ such that
given  $W_{0}$ and $\mathcal{F}_{0}$ as above then
the image of 
$\mathcal{U} \cap \mathcal{G}_{1} \cap \mathcal{G}_{2} (\eta_{0}, W_{0}, \mathcal{F}_{0})$
by the application $S_{0}$ contains a ball of radius $\delta_{0}$ centered on $S_{0}(E)$.
\end{teo}

For $c_{1}$, we take $\mathcal{F}_{1}$, $W_{1}$, and
$\mathcal{G}_{2}(c_{1}, W_{1}, \mathcal{F}_{1})$.
Hereafter, we apply again the above theorem to obtain
the image of a neighborhood of the vector field $E$ by $S_{1}$ which contains a ball of radius 
$\delta_{1}$ centered at $S_{1}(E)$, $B_{\delta_{1}}(S_{1}(E))$.
We repeat this procedure for $\eta_{2}, \dots, \eta_{m}$.

The sets $\mathcal{G}_{2}(c_{i}, W_{i}, \mathcal{F}_{i})$, $i \in \{ 0, \dots, m \}$,
guarantee us there is no interference between one perturbation and the next one and, 
moreover, they are contained in the following set.

\begin{dfn}
We define $\mathcal{G}_{2}(\eta, E, W)$ the set of vector field 
$\tilde{E} \in \mathscr{X}^{\infty}(M)$ such that
$$supp(\| \tilde{E} - E \|) \subset W.$$
\end{dfn}

To finish the proof, we consider the application 
$$S = (S_{0}, \dots, S_{m}): \mathscr{X}_{g}(M) \cap \mathscr{X}^{r}(M) 
\longrightarrow CS(n) \times \dots \times CS(n).$$
The image of $\mathcal{U} \cap \mathcal{G}_{1} \cap \mathcal{G}_{2}(\eta, E, W)$
contains the product $B_{\delta_{0}}(S_{0}(E)) \times \dots \times B_{\delta_{m}}(S_{m}(E))$.
This result, translated to the transversal linear cocycle language means that
a perturbation on the transversal linear cocycle contained on a ball centered at $T$
and radius $min_{i \in \{0, \dots, m\} }\{\delta_{i}\}^{m}$
can be realized with a vector field $\tilde{E}$ in
$\mathcal{U} \cap \mathcal{G}_{1} \cap \mathcal{G}_{2}(\eta, E, W)$.

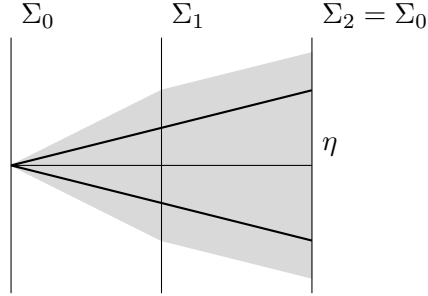
\begin{center}
\begin{tikzpicture}
\filldraw[gray!30] (0,0) -- (2,1) -- (4,1.5) -- (4,-1.5) -- (2,-1) -- cycle;
\draw (0,0) -- (4,0) node[anchor=south west] {$\eta$};
\draw (0,-1.7) -- (0,1.7) node[anchor=south west] {$\Sigma_{0}$};
\draw (2,-1.7) -- (2,1.7) node[anchor=south west] {$\Sigma_{1}$};
\draw (4,-1.7) -- (4,1.7) node[anchor=south west] {$\Sigma_{2}=\Sigma_{0}$};
\draw[thick] (0,0) -- (2,0.5) -- (4,1);
\draw[thick] (0,0) -- (2,-0.5) -- (4,-1);
\end{tikzpicture}
\end{center}

\subsection{ Perturbative theorem along an orbit segment. Proof of theorem \ref{segmento}}

Now we  prove  theorem \ref{segmento}, 
which is the intermediate step for the main result of this section.
First we show how to perturb the vector field $E$.
We take Fermi coordinates along $\eta$ and
consider the application $\alpha : [0,\tau] \times \mathbb{R}^{n} \longrightarrow S(n)$
of class $C^{\infty}$ with support in a neighborhood of
$[0,\tau] \times \{ 0 \}$
and $\tilde{\lambda} : [0,\tau] \times \mathbb{R}^{n} \longrightarrow \mathbb{R}$
such that $\tilde{\lambda}(t,0)\equiv \frac{\lambda}{\tau}$ 
for $t \in [0,\tau]$, $\lambda >0$ and
has its support in a neighborhood of $[0,\tau] \times \{ 0 \}$ with $\tau < r_{inj}$.

Let $\eta$ be an orbit of the Gaussian thermostat,
$\eta(t) = (u(t), v(t))$, and consider
$$
\left\{\begin{array}{rcl}
\tilde{E}_{0} &=& E_{0} + \tilde{\lambda}(t,x)\\\
\tilde{E}_{i} &=& E_{i}+\sum_{j}\alpha_{ij}x_{j} \text{ para } i=1,\dots , n.
\end{array}\right.
$$
The orbit $\eta$ of $(M, g, E)$ is also an orbit of $(M, g, \tilde{E})$.
In fact, 
\begin{eqnarray*}
\left\{
\begin{array}{lcl}
\dot{u} &=& v\\
\\
\frac{D}{dt}v &=& \tilde{E} - Proj_{v}\tilde{E} = E - Proj_{v}E.
\end{array}
\right.
\end{eqnarray*}
We call the intersection of the set of the perturbations of $E$
given as above with  
$E \in \mathcal{G}_{2} (\eta_{0}, W_{0}, \mathcal{F}_{0})$ as $\mathcal{G}_{2}$.
For the proof, we also need some special functions:

\begin{description}

\item[\textbf{$\phi_{\varepsilon}$}:]
Given $\varepsilon>0$, let $\phi_{\varepsilon} : \mathbb{R}^{n} \to [0,1]$  be a function 
of class $C^{\infty}$
such that $\phi_{\varepsilon}(x)=1$ if $x \in [-\frac{\varepsilon}{4},
\frac{\varepsilon}{4}]^{n}$
and $\phi_{\varepsilon}(x) = 0$ if 
$x \notin [-\frac{\varepsilon}{2}, \frac{\varepsilon}{2}]^{n}$.
We also require that there is a fixed $k$ such that
$$
\| \phi_{\varepsilon}(x)x^{*}P(t)x \| \leq k\|P\|_{C^{0}} + \varepsilon k \| P \|_{C^{1}} 
+ \varepsilon^{2} k \| P \|_{C^{2}},
$$
where $P$ is a function that is defined in proposition \ref{proposicaofranks} and depends on the next two functions defined below.
\item[$\bar{h}$:]
The function $\bar{h}: [0,1] \to [0,1]$ is of class $C^{\infty}$ with support 
far from the intersection points, i.e., 
$supp(\bar{h}) \cap (\pi \circ \eta)^{-1} V_{i} = \emptyset$ 
where $V_{i}$ is a neighborhood of
$\mathcal{F}_{i}$ and such that
$$
\int_{0}^{1}(1-\bar{h}(s))ds < \rho.
$$

\item[\textbf{$\delta$}:]
The function $\delta : [0,1] \to [0,+\infty]$ is of class $C^{\infty}$ such that
$\delta(s)=0$ if $|s-\tau| \geq \lambda$
and $\int_{0}^{1}\delta(s)ds =1$.

\end{description}

To prove theorem \ref{segmento} we  need to compute  $d_{0} F \zeta$ ($\sigma = 0$)
but for completness we compute $d_{\sigma} F \zeta$ at any value of $\sigma.$

\begin{prop} \label{proposicaofranks}
Let $F : S(n)^{3} \times S^{*}(n) \times \mathbb{R} \to CS(n)$ given by
$F(\sigma, \lambda) = d\phi^{\tilde{E}}_{1} = X_{1}$, 
where $(\sigma, \lambda) \in S(n)^{3} \times S^{*}(n) \times \mathbb{R}$ and
\begin{itemize}
\item[(i)] $\sigma = (a, b, c; d) \in S(n)^{3} \times S^{*}(n),$
\item[(ii)] $P(t) = \bar{h}(t) [a \delta(t) + b \delta^{'}(t) + c\delta^{''}(t) + d \delta^{'''}(t)],$
\item[(iii)] $\alpha(t,x)=P(t)\phi_{\varepsilon}(x) x_{j},$
\item[(iv)] $\tilde{E}(t, x)_{0} = 
(1 + (\lambda \bar{h}(t) \delta (t) - 1)\phi_{\varepsilon} (x))  E_{0},$
\item[(v)] $\tilde{E}(t, x)_{i} = E_{i} + \sum_{j} \alpha_{ij}x_{j}.$
\end{itemize}
Then there exists $k$ such that
$$
\| d_{(\sigma, \lambda)} F \zeta \| \geq k \| \zeta \| 
\quad \forall \zeta= (a,b,c;d;e) \in S(n)^{3} \times S^{*}(n) \times \mathbb{R}.
$$
\end{prop}

\begin{proof}
Consider the path in $ \mathscr{X}_{g}(M)$ given by,
$$
\gamma: s \to (\sigma, \lambda)  +  s\zeta, 
$$
and the Jacobi equation for $\gamma(s)$ along $\eta$:
$$
\dot{T}_{s} = \mathbb{A}_{s}T_{s},
$$
where
$$
\mathbb{A}_{s}=
\left[\begin{array}{cc}
 0 &  I\\
 K_{s} + B_{s} &  C_{s}
\end{array}\right]
$$
\begin{eqnarray*}
K_{s} &=& K, \\
B_{s} &=& B + sP(t),\\
C_{s} &=& (1 + s\lambda) C.\\
\end{eqnarray*}
Differentiating this equation with respect to $s$, we have 
a differential equation for $Z_{t} = \frac{dT_{s}(t)}{ds}\Big{|}_{s=0}$
$$
\dot{Z} = \mathbb{A}Z + \mathbb{B}X,
$$
where
$$
\mathbb{A}=
\left[\begin{array}{cc}
 0 &  I\\
 K + B &  C
\end{array}\right],
$$
$$
\mathbb{B}=
\left[\begin{array}{cc}
 0 &  0\\
 P(t)  &  \lambda C
\end{array}\right].
$$
We know $Z_{1} = d_{(\sigma, \lambda)}S . \zeta$. This follows from the definition of $Z$:
$$
Z(1) = \frac{d}{dr} T^{r}(1) \Big{|}_{r=0}= \frac{d}{dr} T^{(\sigma, \lambda) + r\zeta}(1) \Big{|}_{r=0}= \frac{d}{dr}(F((\sigma, \lambda) +r\zeta))\Big{|}_{r=0}.
$$
If we write $Z_{t} = T_{t} Y_{t}$ then $T \dot{Y} = \mathbb{B}T$
(motivation: $T_{X}CS(\mathbb{R}^{n})= T( T_{I}CS(\mathbb{R}^{n}))$).
Moreover, $T_{\lambda}(0) \equiv I$ then $Z(0)=0$ and $Y(0)=0$.
Thus,
$$
Y(t)= \int_{0}^{t} (T^{s})^{-1}\mathbb{B}^{s}T^{s}ds
$$
$$
\mathbb{B}(X)=
\left[\begin{array}{cc}
 0 &  0\\
 P(t)  &  \lambda C
\end{array}\right]
= \bar{h}(t) \{  \delta (t)\tilde{A} + \delta^{'}(t) \tilde{B} + \delta^{''}(t) \tilde{C} + \delta^{'''}(t) \tilde{D} + \delta(t) \tilde{F}\}
$$
$$
\tilde{A} =
\left[\begin{array}{cc}
 0 &  0\\
 a   &  0
\end{array}\right],
\tilde{B} =
\left[\begin{array}{cc}
 0 &  0\\
 b  &  0
\end{array}\right],
\tilde{C}=
\left[\begin{array}{cc}
 0 &  0\\
 c &  0
\end{array}\right],
\tilde{D}=
\left[\begin{array}{cc}
 0 &  0\\
 d  &  0
\end{array}\right],
$$
$$
\tilde{F} =
\lambda \sigma
\left[\begin{array}{cc}
 0 &  0\\
 0  & I
\end{array}\right]
$$
Integrating by parts $Y(1)$, we have:
\begin{small}
\begin{eqnarray*}
\int_{0}^{1} T_{s}^{-1} \delta^{'}(s) \tilde{B} T_{s} ds &=& \int_{0}^{1} \delta(s) T_{s}^{-1}[\mathbb{A}\tilde{B} - \tilde{B}\mathbb{A}]T_{s}ds\\
&=& \int_{0}^{1} \delta(s) T_{s}^{-1} \left[\begin{array}{cc}
 b          &  0\\
 \sigma b  &  -b
\end{array}\right]
 T_{s} ds.\\
\end{eqnarray*}
\end{small}

\begin{small}
\begin{eqnarray*}
\int_{0}^{1} T_{s}^{-1} \delta^{''}(s) \tilde{C} T_{s} ds &=& \int_{0}^{1}  \delta^{'}(s) T_{s}^{-1}
\left[\begin{array}{cc}
 c          &  0\\
 \sigma c  &  -c
\end{array}\right]
T_{s} ds\\
&=& \int_{0}^{1} \delta(s) T_{s}^{-1} \Big{(} \mathbb{A}
\left[\begin{array}{cc}
 c          &  0\\
 \sigma c  &  -c
\end{array}\right]
+
\left[\begin{array}{cc}
 c          &  0\\
 \sigma c  &  -c
\end{array}\right]
\mathbb{A} \Big{)} T_{s} ds \\
&=& \int_{0}^{1} \delta(s) T_{s}^{-1}
\Big{(}
\left[\begin{array}{cc}
 0  &  -2c \\
 Kc+cK  &  0
\end{array}\right]
+
\left[\begin{array}{cc}
 0  &  0\\
 Bc+cB  &   0
\end{array}\right]
+
\left[\begin{array}{cc}
 \sigma c     &  0 \\
 \sigma^{2} c &  -\sigma c
\end{array}\right]
\Big{)}
T_{s} ds.\\
\end{eqnarray*}
\end{small}

\begin{small}
\begin{eqnarray*}
\int_{0}^{1} T_{s}^{-1} \delta^{'''}(s) \tilde{D} T_{s} ds &=&
\int_{0}^{1} \delta^{'} (s) T_{s}^{-1}
\Big{(}
\left[\begin{array}{cc}
 0  &  -2d \\
 Kd+dK  &  0
\end{array}\right]
+
\left[\begin{array}{cc}
 0  &  0\\
 Bd+dB  &   0
\end{array}\right]
+
\left[\begin{array}{cc}
 \sigma d     &  0 \\
 \sigma^{2} d &  -\sigma d
\end{array}\right]
\Big{)}
T_{s} ds \\
&=& \int_{0}^{1} \delta (s) T_{s}^{-1}
\Big{(}
\left[\begin{array}{cc}
 Kd+3dK  & 0 \\
   0 &  -3Kd - dK
\end{array}\right] + \\
&&
+
\left[\begin{array}{cc}
 Bd+3dB &  0\\
 2\lambda \{ (K+B)d+d(K+B) \} & -3Bd-dB
\end{array}\right]
+
\left[\begin{array}{cc}
 \sigma^{2} d  & 0  \\
 \sigma^{3} d  & -\sigma^{2} d
\end{array}\right]
\Big{)}
T_{s} ds.
\end{eqnarray*}
\end{small}

The matrices in
$T_{I}^{v}\text{CS}(\mathbb{R}^{n})$ can be written as:
\begin{displaymath}
\left(\begin{array}{cc}
 \beta & \gamma \\
 \alpha & vI-\beta^{*}
\end{array}\right)
\end{displaymath}
where $\alpha$ and $\gamma$ are symmetrical and $\beta$ has no restrictions.
\begin{eqnarray*}
\alpha &=& a + \sigma b + (K+B)c + c(K+B) + \sigma^{2}c + 2\sigma\{ (K+B)d + d(K+B)\} + \sigma^{3}d ,\\
\beta  &=& b + \sigma c + (K+B)d + 3d(K+B) + \sigma^{2}d, \\
\gamma &=& -2c, \\
v      &=& \lambda \sigma,
\end{eqnarray*}
The matrix $\beta$ decompose itself in $\beta = \beta_{sim} + \beta_{asim}$ where
$\beta_{sim}$ is a symmetric and 
$\beta_{asim} = d(K+B) - (K+B)d$ is antisymmetric. 
The following lemma, proved in \cite{contreras}, show that $\beta_{asim}$ 
is determined by $d$. Moreover, $\beta_{sim}$, $\alpha$, and $\gamma$ 
are determined by $b$, $a$, and $c$, respectively. 

\begin{lm} \label{inversao}
Let $W$ a symmetric matrix and consider $L_{W} : S^{*}(n) \to AS(n)$
given by $L_{W} (d) = Wd-dW$. Suppose the eigenvalues $\lambda_{i}$ of $W$
are all distinct. Then for all $f \in AS(n)$
there exists $d \in S^{*}(n)$ such that $L_{W}(d)=f$ and
$$
\| f \| \leq \frac{\| d \|}{min_{i \neq j}|\lambda_{i}-\lambda_{j}|}.
$$
\end{lm}
\end{proof}

To finish the argument we will need the following lemma which shows there exists $k$ such that $\| Z_{1} \| \geq  k\| \zeta \|$ and whose proof can be found in 
\cite{contreras}.

\begin{lm} \label{lemafranks}
Let $\mathscr{N}$ a connected Riemannian manifold with dimension $m$
and $F:\mathbb{R}^{m} \to \mathscr{N}$ a smooth application such that 
$$
| d_{x}F(v)| \geq a > 0 \qquad \forall (x,v) \in T\mathbb{R}^{m} \quad com \quad |v|=1 \quad e \quad |x| \leq r
$$
then for all $0 < b < ar$,
$$
\{ \omega \in \mathscr{N} | d(\omega, F(0))<b \} \subseteq F( \{ x \in \mathbb{R}^{m} | |x|<\frac{b}{a} \} ).
$$
\end{lm}

We will show the image of $\mathcal{U}_{0}$ by $S$ contains a ball in $CS(n)$ 
centered at $S(g)$ and radius $r=r(g, \mathcal{U})$.
Consider the application $G: \mathbb{R}^{2n(n+1)} \to \mathscr{X}^{r}(M)$ 
where $G(\sigma, \lambda)=(\tilde{E})$.
The following diagram is commutative
 \begin{center}
\begin{tikzpicture}
\node(r) at (0,0) {$B(0, k_{5}^{-1}r) \subset \mathbb{R}^{2n(n+1)}$};
\node(s) at (4,0) {$\mathscr{X}^{r}(M)$};
\node(t) at (4,-2) {$CS(n)$};
\draw[->] (r.east) -- (s.west) node[midway, above] {$G$};
\draw[->] (r.south) --(t.west) ;
\draw[->] (s.south) -- (t.north) ;
\node at (4.3,-1) {$S$};
\node at (1.5,-1.5) {$F$};
\end{tikzpicture}
\end{center}

The proposition \ref{proposicaofranks} and the lemma \ref{lemafranks} show that
$B(S(E), r) \subset F(B(0, k_{5}^{-1}r))$
but we need $B(S(E), r) \subset S(\mathcal{U}_{0})$.
To obtain this, is is enough to show 
$G( B(0, k_{5}^{-1}r) ) \subset \mathcal{U}_{0}$.
In fact,
$\| \tilde{E} - E \|_{C^{1}} < \varepsilon$
therefore $\tilde{E} \in \mathcal{V}_{0} \subset \mathcal{U}_{0}$ and we have the result.

\section{The Gaussian thermostat as a Weyl flow. Proof of theorems D and D'}\label{specialconnection3}
In this section we  describe the linear connection $\hat{\nabla}_{X}Y$ introduced in \cite{w1} and defined as
$$
\hat{\nabla}_{X}Y = \nabla_{X}Y - \langle X, Y \rangle E + \gamma(Y)X + \gamma(X)Y.
$$

\begin{prop}
The linear connection $\hat{\nabla}$ is symmetric and it is non compatible with the metric.
\end{prop}

\begin{proof}
\begin{description}
\item[Symmetry]
\begin{eqnarray*}
\tilde{T}(X,Y)&=&\hat{\nabla}_{X}Y - \hat{\nabla}_{Y}X - [X,Y]\\
&=& \nabla_{X}Y - \langle X, Y \rangle E + \gamma(Y)X \\
& & +\, \gamma(X)Y - \nabla_{Y}X + \langle X, Y \rangle E - \gamma(X)Y - \gamma(Y)X - \nabla_{X}Y + \nabla_{Y}X\\
&=& 0.
\end{eqnarray*}

\item[Non compatibility with the metric]

\begin{eqnarray*}
Z\langle X,Y \rangle &=& \langle \nabla_{Z}X,Y\rangle + \langle X, \nabla_{Z}Y\rangle\\
&=& \langle \hat{\nabla}_{Z}X +  \langle Z, X \rangle E - \gamma(X)Z - \gamma(Z)X,Y\rangle\\
& & + \,
\langle X, \hat{\nabla}_{Z}Y +  \langle Z, Y \rangle E - \gamma(Y)Z - \gamma(Z)Y\rangle\\
&=& -2\gamma(Z) \langle X, Y \rangle +  \langle \hat{\nabla}_{Z}X,Y\rangle + \langle X, \hat{\nabla}_{Z}Y\rangle.\\
\end{eqnarray*}

\end{description}
\end{proof}

%---
The reparametrization by arc length of the geodesic flow of $\hat{\nabla}$ is called  {\it Weyl flow}. The folowing proposition shows the Weyl flow 
restricted to $SM$ is a Gaussian thermostat.

\begin{prop}
The Weyl flow is the Gaussian thermostat $(M,g,E)$ when 
we consider the restriction of it to $SM$.
\end{prop}

\begin{proof}
Let $u$ a geodesic of $\widehat{\nabla}$ and 
$\psi : \mathbb{R} \to \mathbb{R}$ the reparametrization by arc length of $u$  such that
$$
\eta(t) = u(\psi (t)) \quad \text{e} \quad u(s) = \eta (\psi^{-1} (s)). 
$$
Let $\frac{du}{ds}=w$ e $\frac{d \eta}{dt} = v$ then 
\begin{eqnarray*}
w &=& v \frac{dt}{ds} = v \| w \|, \\
\frac{d\| w \| }{ds}&=& -\gamma (w)\| w \|, \\
\frac{d \| w \|}{dt} &=& -\gamma (w) \frac{\| w \|}{\| w \|} = -\gamma (w),
\end{eqnarray*}
and
\begin{eqnarray*}
0 &=& \widehat{\nabla}_{w}w \\
&=& \nabla_{w}w + 2\gamma(w)w - \| w \|^{2}E \\
&=& \nabla_{v\frac{dt}{ds}}v\frac{dt}{ds} + 2\gamma\Big{(}\frac{dt}{ds}v\Big{)}\frac{dt}{ds}v - \Big{\|} \frac{dt}{ds}v \Big{\|}^{2}E \\
&=& \frac{dt}{ds} v\Big{(}\frac{dt}{ds}\Big{)}  v + \Big{(}\frac{dt}{ds}\Big{)}^{2}\nabla_{v}v +  
    2\Big{(}\frac{dt}{ds}\Big{)}^{2}\gamma(v)v - 
    \Big{(}\frac{dt}{ds}\Big{)}^{2} \| v \|^{2}E\\
&=& - \Big{(}\frac{dt}{ds}\Big{)}^{2} \gamma(v)  v + \Big{(}\frac{dt}{ds}\Big{)}^{2}\nabla_{v}v +  
    2\Big{(}\frac{dt}{ds}\Big{)}^{2}\gamma(v)v - 
    \Big{(}\frac{dt}{ds}\Big{)}^{2} \| v \|^{2}E.\\
\end{eqnarray*}

In terms of the Riemannian connection, the equation can be written as

$$\left\{\begin{array}{rcl}
\dot{x}&=&v \\
%\\
\nabla_{v}v &=& E - \gamma(v) v.\\
\end{array}\right.$$

This means $\eta$ is an orbit of $(M,g,E)$ when $\langle v, v \rangle=1$.
\end{proof}

%---
%Podemos fazer geometria: curvatura, blabla
We define the curvature tensor associated to $\hat{\nabla}$ as
$$
\tilde{R}(X,Y)= \hat{\nabla}_{X}\hat{\nabla}_{Y}-
\hat{\nabla}_{Y}\hat{\nabla}_{X}+\hat{\nabla}_{[X,Y]},
$$
and the sectional curvature
$$
\tilde{K}(X,Y)=\langle \tilde{R}(X,Y)X , Y \rangle.
$$

%relacao entre as curvaturas sectionais: ainda nao encontrei algo facil

%Jacobi equation
%Let $\eta(s)$ be a geodesic of the connection $\hat{\nabla}$ and $u(s,a)$
%a family of geodesic, $u(s,0) = \eta(s)$, where $s$ is the arc length parameter
%on a geodesic, and a is a real parameter, taken from a small interval around $0$.
%Define the unit velocity field $v$ and the Jacobi field $J$ by
%\begin{eqnarray*}
%v&=&\frac{\partial u}{\partial s}\\
%J&=& \frac{\partial u}{\partial a}\Big{|}_{a=0}.
%\end{eqnarray*}
%
%Using the following two identities valid in general
%
%\begin{eqnarray*}
%\hat{\nabla}_{v}{J} - \hat{\nabla}_{J}v &=&\tilde{T}(v,J)\\
%\hat{\nabla}_{J}\hat{\nabla}_{v}v - \hat{\nabla}_{v}\hat{\nabla}_{J}v &=& \tilde{R}(v,J)v.
%\end{eqnarray*}
%
%and introducing $\chi = \hat{\nabla}_{v}J$ we get the Jacobi equation
%
%\begin{eqnarray*}
%\xi &=&  \hat{\nabla}_{v} J = \hat{\nabla}_{J} v\\
%\hat{\nabla}_{v}\chi &=&  - \tilde{R}(v,J)v.
%\end{eqnarray*}
%
%Over the quotient $\hat{T}SM$ we write the Jacobi equation as
%
%$$
%\hat{\nabla}_{v} \chi =  - \tilde{R}(v,J)v.
%$$
Let
$f(s,u)=\pi\circ\phi^{t}(z(u))$ be a family of geodesic of the connection $\hat{\nabla}$ parametrized by $u$. Then the Jacobi field for the geodesic flow $\hat{\nabla}$  is given by,
\begin{eqnarray*}
w(s) &=&\frac{\partial }{\partial s} f (s,u),\\
J(s) &=&\frac{\partial}{\partial u} f (s,u),\\
\dot{J}(s)&=&\widehat{\nabla}_{w(s)} J(s),\\
\ddot{J}(s)&=&-\widehat{R}(w,J)w.\\
\end{eqnarray*}
Consider now the reparametrization of that family by arclength 
$\tilde{f}(t,u)=f(\psi(t,u),u)$ where $\psi(t,u)$ is the reparametrization by arc length of $\phi_{t}(z(u))$. The Jacobi fied $\hat{J} = \frac{\partial}{\partial u}\tilde{f}$ satisfies

\begin{eqnarray*}
\hat{J}&=& \frac{\partial \psi}{\partial u}w + J,\\
\dot{\hat{J}}&=& v(\frac{\partial \psi}{\partial u})w+\frac{1}{\|w\|}\dot{J},\\
\ddot{\hat{J}}&=& v(v(\frac{\partial\psi}{\partial u}))w + \frac{1}{\|w\|}\gamma(v)\dot{J} + \frac{1}{\|w\|^{2}}\ddot{J}\\
&=& v(v(\frac{\partial\psi}{\partial u}))w + \frac{1}{\|w\|}\gamma(v)\dot{J} - \frac{1}{\|w\|^{2}}\widehat{R}(w, J)w.
\end{eqnarray*}
Therefore, over $\widehat{T}_{\theta}SM$, we have the Jacobi field for the Weyl flow
\begin{eqnarray*}
\hat{J}&=&J,\\
\dot{\hat{J}}&=& \frac{1}{\|w\|}\dot{J},\\
\ddot{\hat{J}}&=& \frac{\gamma(v)}{\|w\|}\dot{J} - \frac{1}{\|w\|^{2}}\widehat{R}_{a}(w, J)w.
\end{eqnarray*}

%---

%diferencial do fluxo e equação de jacobi

%----

\subsection{Proof of geometrical theorem (Theorems D and D')}

The aim now is to relate sectional curvature with the dynamics of a Gaussian thermostat.
We start with some results proved in \cite{w1}.

\begin{teo}[\cite{w1}]
If the sectional curvature of the Weyl structure is negative everywhere then the W-flow has dominated splitting, with exponential growth/decay of volumes.
\end{teo}
In particular, for the case of surfaces follows that
\begin{teo}[\cite{w1}] For surfaces, if the  curvature of the Weyl structure is negative everywhere then the W-flow is a transitive Anosov flow.
\end{teo}

In theorem D and D'   the hypothesis of negative (sectional) curvature are relaxed.

\noindent{\it Proof of theorem D:}
Consider the %Lyapunov 
function $\mathcal{L}: \hat{T}SM \longrightarrow \mathbb{R}$ defined as
$$
\mathcal{L}(\xi)= \langle \xi_{h}, \xi_{v} \rangle, 
$$
where $\xi_{h}$ and $\xi_{v}$ are the horizontal and vertical components of $\xi$, respectively.

%Using the fact that the connection $\hat{\nabla}$ 
%is metric we get
%$$
%\frac{d}{dt}\mathcal{L} (\xi)= \langle \hat{\nabla}_{v} \xi_{h} , \xi_{v} \rangle + \langle \xi_{h}, \hat{\nabla}_{v}\xi_{v} \rangle.
%$$

We can write $\xi \in \hat{T}SM$ as $\xi=(\xi_{h}, \xi_{v})$ e $\xi_{h}=\hat{J}(t)$ e $\xi_{v}=\dot{\hat{J}}(t)$ where $J$ is a jacobi field. So the jacobi equation for the Weyl flow over the quotient $\hat{T}SM$ give us 

%$$
%\frac{d}{dt}\mathcal{L} (\xi)= \langle  \xi_{v}, \xi_{v} \rangle + \langle \tilde{T}(v,\xi_{h}), \xi_{v} \rangle + \langle \xi_{h}, \tilde{R}(v,\xi_{h})v \rangle.
%$$

%then using $\tilde{T}(X,Y) = \gamma(Y)X-\gamma(X)Y$ 

$$
\frac{d}{dt}\mathcal{L} (\xi)= \langle \xi_{v}, \xi_{v} \rangle + \gamma(v) \langle \xi_{h}, \xi_{v} \rangle - \langle \xi_{h}, \hat{R}(v,\xi_{h})v \rangle. 
$$

Fix $k>0$. Defining the positive cone 
$C_{k}^{+}(\pi (\xi)) = \{ \xi \in \hat{T}_{\pi (\xi)}SM | \mathcal{L}(\frac{\xi}{\| \xi \|}) \geq  k \}$. This defines a cone field over $SM$. 
The application $\mathcal{L}$ is continuous and monotonically increasing over $C^{+}_{k}$ then we can conclude that positive cones are sent in the interior of positive cones by $T^{t}$. 

Analogously, using the reversibility of the flow we  get negative cones
$C_{k}^{-}(\pi (\xi)) = \{ \xi \in \hat{T}_{\pi (\xi)}SM |  \mathcal{L}(\frac{\xi}{  \| \xi \|}) \leq  -  k \}$ are sent in the interior of negative cones by $T^{-t}$. This caracterizes dominated splitting for the flow.
\qed

% 
% Denote by $Per(X)$
% the set of periodic orbits of $X$, and 
% $Sing(X)=\{ p \in M | X(p)=0 \}$, for a given $X \in \mathscr{X}^{r}(M)$.
Using the following result, the proof of  theorem D' holds.

\begin{teo}(\cite{aubin})
Let $M$ be a three dimensional closed manifold and  $\Lambda$ be a non-singular  compact invariant set for $X \in \mathscr{X}^2(M)$ with a dominated
splitting such that  all periodic trajectory  points  are hyperbolic
saddles, then $\Lambda = \tilde{\Lambda} \cup \mathcal{T}$, where
$\tilde{\Lambda}$ is hyperbolic and  $\mathcal{T}$ 
is a finite union of irrational
tori.
\end{teo}

%\begin{proof}
%In dimension $2$ dominated splitting and exponential growth and decay of volume implies hyperbolicity.
%\end{proof}

\end{document}